\documentclass{article}

\input{packages_and_shortcuts}

\begin{document}

\title{
Distortion estimates for barycentric coordinates on Riemannian simplices}
\author{Stefan W. von Deylen\thanks{formerly Freie Universit\"at Berlin,
									\texttt{buero@SW.vonDeylen.net}.}
\and David Glickenstein\thanks{University of Arizona,
		                            \texttt{glickenstein@math.arizona.edu}.}
	\and Max Wardetzky\thanks{Universit\"at G\"ottingen,
		                            \texttt{wardetzky@math.uni-goettingen.de}.}}
\date{
\today}
\maketitle

\begin{abstract}
\noindent
We define barycentric coordinates on a Riemannian manifold using
Kar\-cher's center of mass technique applied to 
point masses for $n+1$ sufficiently close points, determining an $n$-dimensional 
Riemannian simplex defined as a ``Karcher simplex.'' Specifically,
a set of weights is mapped to the Riemannian center of mass
for the corresponding point measures on the manifold with the given weights.
If the points lie sufficiently close and in general position,
this map is smooth and injective, giving a coordinate chart. We 
are then able to compute first and 
second derivative estimates of the coordinate chart. These estimates 
allow us to compare
the Riemannian metric with the Euclidean metric induced on a simplex
with edge lengths determined by the distances between the points. We
show that these metrics differ by an error that shrinks quadratically
with the maximum edge length. With such estimates, one can deduce 
convergence results for finite element approximations of problems on
Riemannian manifolds.
\end{abstract}

\vspace{1ex}
{\sloppypar
\tableofcontents
}


\section{Introduction}

There are two major approaches to numerical computations on a Riemannian manifold: (1) mapping the smooth manifold to 
a triangulated manifold and performing computations on Euclidean simplices (see, e.g., \cite{Dziuk88,Bartels10,Holst12}), or (2) performing each computation in a natural chart, 
such as geodesic (normal) coordinates (e.g., see \cite{Muench07,Huper04,Cruzeiro06}). The advantage
of using triangulations is that they provide a global description, independent of coordinates.
Normal coordinates, in contrast, allow for interior estimates
on metric distortion and curvature 
inside coordinate charts, but can be difficult to work with globally due to coordinate changes. 
Here we present a matrimony between these two approaches using Karcher's center of mass
technique. 

The goal of the present work is to give a comprehensive treatment of the interior
estimates from the perspective of geometric analysis. 
We will not provide details for applications, most of which follow from our estimates using 
standard principles from numerical analysis.  
Many possible applications are given in \cite{Deylen14}.---For a short overview of related
work, see sec.~\ref{sec:relatedWork}.

\subsection{Barycentric coordinates}
Riemannian barycentric coordinates are based on the notion of barycentric coordinates in Euclidean space.
Let $\stds := \conv(e_0,\dots,e_n) \subset \R^{n+1}$ denote the standard simplex, and let 
$p_0,\dots,p_n$ be points in Euclidean space $\R^m$. The
barycentric coordinates of the (possibly degenerate)
simplex $s := \conv(p_0,\dots,p_n)$ are defined by
\[
  x: \stds \to s \ , \quad
	x(\lambda) =  \sum_{i=0}^n \lambda^i p_i \ .
\]
Notice that since the $\lambda^i$ sum to $1$,
the point $x(\lambda)$ is the minimizer of
the function 
\[
  E_\lambda(a) := \sum_{i=0}^n \lambda^i \absval{a - p_i}^2.
\] 

This construction can be generalized to Riemannian manifolds. To do this we will need a notion of 
a Riemannian simplex.
Let $p_0,\dots,p_n$ be distinct points in a convex ball $B$ of a complete Riemannian
manifold $(M,g)$ of dimension $m$. Let the pairwise geodesic distances between these given points satisfy $\dist_g(p_i,p_j) \leq h$ for some $h>0$.
Let $g^e$ be the (unique if it exists) flat metric on the standard simplex $\stds
\subset \R^{n+1}$ such that the induced edge lengths of $\stds$ are given by the geodesic distances $\dist_g(p_i,p_j)$. Below we discuss conditions for when such a flat metric exists. 
Suppose that for
some $\theta > 0$, $g^e$ gives a ``$(\theta,h)$-full'' simplex in the sense that
the induced volume satisfies $n! \vol_{g^e}(\stds) \geq \theta h^n$. Then for 
fixed $\theta$ and sufficiently small $h$, the map
(introduced in \cite{Grove73})
\begin{align*}
	x: \stds & \to M,
\quad \text{defined via}\quad
	   x(\lambda) = \argmin_{a \in M} \sum_{i=0}^n \lambda^i \dist^2_g(p_i,a),
\end{align*}
is a bijection between $\stds$ and a subset $s$ of $B$, called the
\emph{Karcher simplex} with respect to vertices $p_i$.

Now consider a global triangulation of $M$ by Karcher simplices.
Notice that if $\lambda \in \stds$ lies in the facet opposite
to the $i$'th vertex
of $\stds$, then its component $\lambda^i$ is
zero and $x(\lambda)$
does not depend on $p_i$. Therefore, 
the flat (Euclidean) simplices can be glued together to give a piecewise flat manifold that is homeomorphic to $M$, and $x$ can be extended to a provide a global homeomorphism.
The non-degeneracy of these simplices will be assumed in our setting,
but is currently investigated in more detail in \cite{Dyer14}
(see also the previous work in \cite{Boissonnat11}).

We provide estimates on how well the piecewise Euclidean structure, $g^e$, of this piecewise flat manifold approximates the smooth Riemannian one, $g$, of $M$. Notice that on each Karcher simplex the map $x$ pulls back the Riemannian metric $g$ on $M$ to $\stds$. We derive estimates for both first and second derivatives of
$x$ on a simplex. We then derive estimates for the difference $x^*g-g^e$ as well as for $\nabla^e x^* g$, where $\nabla^e$ denotes the covariant derivative induced by $g^e$.

\subsection{Main results}
\label{sec:main}
The tangent space $T_\lambda\stds$ of $\stds$ can be identified with 
$\{v \in \R^{n+1}:\sum{v^i}=0\}$. Let $p_0,\dots,p_n\in M $ be distinct points inside a convex ball of radius $h$ in the Riemannian manifold $(M,g)$ of dimension $m$.
A set is convex if each pair of points has a unique shortest geodesic that lies
entirely in that set.

\begin{definition}
For an arbitrary fixed $\lambda \in \stds$ and $v, w \in T_\lambda\stds $ define
\[
		g^e(v,w) = - \frac 1 2 \sum_{i,j=0}^n{\dist^2(p_i, p_j)v^iw^j}.	
\]
\end{definition}
In Section \ref{sec:EuclidSimplices}
we discuss conditions for when $g^e$ yields a metric, i.e., when it is positive definite. For now suppose it does. In the following we are interested in how well $x^*g$ is approximated by $g^e$. 

In our estimates we require a definition of fullness that quantifies how ``thin'' a simplex can become with respect to some Riemannian metric.
\begin{definition}\label{def:fullness}
  A $n$-simplex $s$ with Riemannian metric $g$ is \emph{$(\theta, h)$-full}
  if all edges have length less than or equal to $h$ and
  \[
    n! \vol_g(s) \geq \theta h^n,
  \]
  where $\vol_g(s)$ is the Riemannian volume.
\end{definition}

Among Euclidean simplices, the maximal $\theta$ is attained for the equilateral simplex,
which shows $\theta \leq \sqrt{n+1} / 2^{n/2}$.



Now let us look at $x$ from above for the simplest case $M = \R^{n+1}$. Here,
one has $dx(v) = \sum v^i X_i|_{x(\lambda)}$ for every $v\in T_\lambda \stds$, where $X_i$ is the vector field on $\conv(p_0,\dots,p_n)$ defined by $X_i = \frac 1 2 \grad \dist^2 (\argdot,p_i)$. This motivates the following definition in the Riemannian setting.
\begin{definition}
For every $v\in T_\lambda \stds$ define
	\[
  	\sigma (v) =	\sum_{i=0}^n v^i X_i|_{x(\lambda)},
	\] 
	where $X_i = \frac 1 2 \grad \dist^2 (\argdot,p_i)$ are vector fields on $x(\stds)\subset M$.
\end{definition}

The following theorem quantifies how much $\sigma$ deviates from $dx$. Before stating this result, we require some additional notation. We use $\nabla d x$ to denote the Hessian of $x$ with respect to the Levi-Civita connection on $(M,g)$ and the flat connection induced by $g^e$ on $\stds$. We use $\inorm[\infty]{R}$
and $\inorm[\infty]{\nabla R}$ to denote the supremum over the manifold $M$ of the usual pointwise 2-norm of the Riemannian curvature tensor and its covariant derivative (with respect to the metric $g$), respectively.

	Throughout, we let $C_0 := \inorm[\infty] R$, $C_1 := \inorm[\infty]{\nabla R}$, and we let $\iota_g$ denote the injectivity radius of $(M,g)$.
	Additionally, when working in a ball of radius $h$, we use $C_{0,1} := C_0 + h C_1$.  
	%
%

\begin{theorem}\label{thm:main1}
	\label{thm:EstimateForx}
	There exist constants $\alpha=\alpha(n,\theta,C_0,C_1)$, $\beta=\beta(n,C_0)$,
	 and $\gamma=\gamma(n,\theta,C_0,C_1)$ 
	such that if $h < \alpha$ and $(\stds, g^e)$ is a $(\theta,h)$-full simplex then
  \begin{equation}
    \bigabsval[g]{dx(v) - \sigma(v)}  \leq
			\beta \,h^2 \absval[g^e]{v} ,
	\end{equation}
	and
	\begin{equation}
	  \absval[g]{\nabla dx(v,w)} \leq \gamma \,h \absval[g^e] v \,
	   	\absval[g^e] w,
  \end{equation}
  for	tangent vectors $v,w \in T_\lambda \stds$ at any $\lambda \in \stds$.	 
	   
\end{theorem}

These estimates can also be interpreted as estimates on the difference between the flat metric
$g^e$ and the pullback of the metric $g$ by the map $x$. We use $\nabla^e$ to denote the flat connection
on $(\stds, g^e)$.

\begin{theorem}
	\label{thm:comparisongge1}
	There exist constants $\alpha=\alpha(n,\theta,C_0,C_1)$, $\beta=\beta(n,\theta,C_0)$,
	 and $\gamma=\gamma(n,\theta,C_0,C_1)$ such that 
	if $h < \alpha$ 	and $(\stds, g^e)$ is a $(\theta,h)$-full simplex then 
	\begin{equation}
		\label{eqn:comparisongge1}
		\absval{(x^* g-g^e)(v,w)} \leq 
		  \beta \,h^2	\absval[g^e] v \, \absval[g^e] w,
	\end{equation}
	and
	\begin{equation}
	  \label{eqn:comparisonnablgge}
	  \absval{\nabla^e x^*g (u,v,w)} \leq 
	  \gamma\, h \absval[g^e] u \, \absval[g^e] v \, \absval[g^e] w
	\end{equation}
	for tangent vectors
	$u,v,w \in T_\lambda \stds$ at any $\lambda \in \stds$.
\end{theorem}

In loose terms, eqn.~\eqref{eqn:comparisongge1} gives second-order control
over $x$'s first derivatives, whereas
eqn.~\eqref{eqn:comparisonnablgge} gives first-order control over the second
derivatives.---
These theorems follow immediately from Theorems
\ref{thm:EstimateFordx},
\ref{thm:comparisongge},
\ref{thm:EstimateForNabladx}, and
\ref{thm:EstimateForChristoffelSymbols}
and are proven in Section \ref{sec:Approximation}.

\subsection{An application: The Poisson equation}
\label{sec:poisson}
Our estimates allow for proving finite element approximation results on Riemannian manifolds. Previous work in this area has restricted largely to hypersurfaces in $\R^{n+1}$ and submanifolds of Euclidean spaces (e.g., \cite{Dziuk13,Bertalmio01,Hildebrandt11}). Often these works use embedded polyhedra to construct finite elements, together with shortest distance maps that map Euclidean simplices to the manifold by assigning each point on the polyhedron to the closest point on the manifold, an approach dating back to \cite{Dziuk88}. In many cases, the barycentric finite elements and barycentric coordinates described here can be used in place of this construction.

As an example, consider the the Poisson equation on a closed Riemannian manifold, i.e.,
\begin{equation}
		\label{eq:poisson}
\laplace_g u=f ,
\end{equation}
where $\laplace_g$ is the Riemannian Laplacian.
Equipped with some space of Riemannian finite elements $V_{h}$ to be determined, consider the Galerkin approximation $u_h \in V_h$
that solves 
\begin{equation}
		\label{eq:finiteelementpoisson}
  \int_M g\left(d u_h, d v\right) \dvol_g = \int_M fv \dvol_g
  \qquad \text{for all $v \in V_h$},
\end{equation}
where $\dvol_g$ is the volume form for the metric $g$. Then for the solutions $u$ and $u_h$ of Equations \eqref{eq:poisson} and \eqref{eq:finiteelementpoisson}, respectively, the following estimates are analogues of the usual method of proving convergence of a Galerkin approximation:
(a) C\'{e}a's Lemma:
\[
\inorm[\Sob^1]{ u-u_{h}} \leq c \inf_{v\in V_{h}} \inorm[\Sob^1]{u-v}.
\]
and (b) an interpolation estimate:
\[
\inf_{v\in V_{h}} \inorm[\Sob^1]{u-v} \leq c h \inorm[\Sob^2] u
\qquad \text{for all $u \in \Sob^2$.}
\]

Given these, one can derive the estimates (in case
the Poisson problem is $\Sob^2$-regular)
\[
  \inorm[\Sob^1]{u-u_{h}}
  \leq c_1 \inf_{v\in V_{h}} \inorm[\Sob^1]{u-v}
  \leq c_2 h \inorm[\Sob^2] u
  \leq c_3 h \inorm[\Leb^2]f,
\]
proving convergence of $u_{h}$ to $u.$ 

The piecewise flat manifolds arising from our construction using Karcher simplices naturally provide piecewise linear finite element spaces.
While C\'{e}a's lemma is 
relatively straightforward in this setting, 
the interpolation lemma
requires control over
second derivatives of the 
map $x$ on each simplex. 
Theorem~\ref{thm:main1} provides the requisite estimate.


Similar to this example, the results of our Theorems \ref{thm:comparisongge} and
\ref{thm:EstimateForChristoffelSymbols} pave a natural way to carry over the
existing Finite Element approaches for heat flow \cite{Dziuk13}, Hodge decomposition
\cite{Polthier00}, a weak shape operator \cite{Hildebrandt06} and minimal submanifolds
\cite{Pinkall93} from embedded surfaces to arbitrary smooth manifolds with an
appropriate triangulation. These applications have been detailed in
\cite{Deylen14}.

\subsection{Related approaches using the Karcher Mean}
\label{sec:relatedWork}
Usage of the barycentric coordinates is mostly split into
three non-intersecting communities from statistics, geometry
and numerics.

The Karcher mean has been a standard tool in statistical environments,
where interpolation in nonlinear matrix spaces is very frequent,
for a relatively long time \cite{Moakher05}. We only refer to
overview articles \cite{Afsari11,Kendall13}
and the references therein. To these researchers, Karcher means
are an averaging method for given data points in a manifold.

Coming from the purely geometric perspective of triangulations,
Wintraecken \emph{et al.} have dealt with the distortion of barycentric coordinates
using Topogonov's angle theorem, arriving at a similar result for $x^*g - g^e$
as our eqn.~\eqref{eqn:comparisongge1} \cite{Dyer14}, butwith  no analogue for
eqn.~\ref{eqn:comparisonnablgge}.
Their focus now goes into the direction of well-definedness
and continuity of the coordinates across simplex boundaries \cite{Dyer16}.

In the theory of finite elements, the group around Grohs and Sander considers problems
about maps \emph{into} manifolds, whereas our results have their primary use when
mapping from a manifold into, say, the real numbers.
>From this perspective, barycentric coordinates are an interpolation
procedure for vertices which are themselves subject to an optimization
problem \cite{Sander12, Grohs13}. Important questions from this point of view
include higher-order interpolation \cite{Sander13}, test vector fields \cite{Sander16},
and the rigorous treatment of other problems 
such as nonlinear elliptical energies or function-space gradient flows
\cite{Hardering15}.

The definition of a Karcher simplex would carry over to length spaces
in a natural way, but we are not aware of any work in this direction yet.

\subsection{Structure of the paper}
The main part of this article will be structured in three parts:
First an overview of flat \textit{simplex} metrics and barycentric coordinates on
them, as we rely on several facts for them that are seldom found together
in one reference. This also includes the tangent space simplices which
we need for the last main part. Next comes the \textit{Riemannian} geometry part,
with a more thorough introduction of the main construction than we could
give in the introduction. In the final part we combine the Euclidean simplex estimates with
smooth Riemannian geometry estimates to bound the distortion
of the barycentric coordinate map to the Karcher simplex.

\begin{acknowledgement}
The authors would like to thank Konrad Polthier and Ulrich Brehm.
DG is partially supported by NSF grant DMS 0748283. SvD has been supported by
Deutsche Forschungsgemeinschaft (DFG) via Berlin Mathematical School.
\end{acknowledgement}

\section{Flat Metrics and Barycentric Coordinates}
\label{sec:flatMetrics}

Before considering general Riemannian metrics, we summarize a few facts about Euclidean simplices. 
The metric of a Euclidean simplex is uniquely determined by the length of its edges. Not every system of edge lengths, however, gives rise to a Euclidean simplex -- 
even if the triangle inequality is satisfied. 
Consider, for example, 
the situation in Figure~\ref{fig:tooThinTriangle}. 
In this section we discuss the existence of flat metrics from a given set of edge lengths.
\begin{figure}
	\begin{center}
	\begin{minipage}{3cm}
	\includegraphics[scale=0.3]{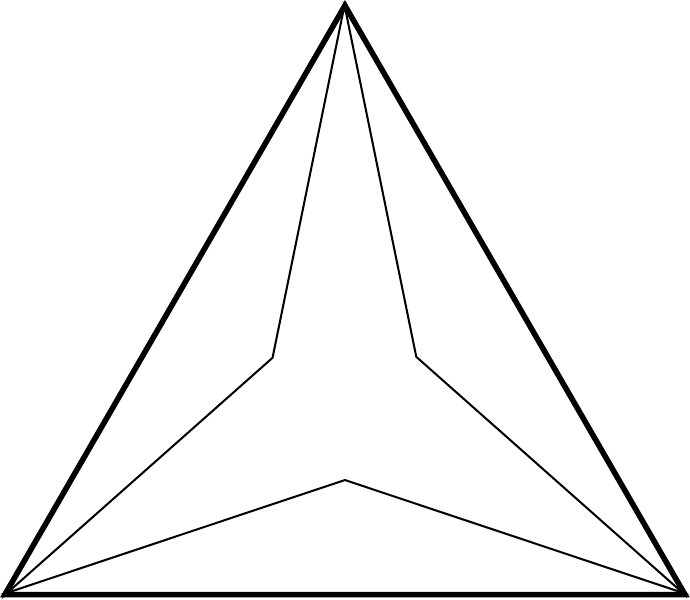}
	\end{minipage}
	\begin{minipage}{8cm}
	\caption{One equilateral triangle and three very thin triangles cannot form
	the boundary of a Euclidean tetrahedron.}
	\label{fig:tooThinTriangle}
	\end{minipage}
	\end{center}
\end{figure}

\subsection{Parametrizations of Euclidean simplices}
\label{sec:EuclidSimplices}

Before defining general simplices, we introduce two useful special ones.
\begin{definition}
  Let $e_0, \ldots, e_n$ be the standard basis of $\R^{n+1}$. The \emph{standard simplex} $\stds \subset \R^{n+1}$ 
  is the convex hull of the points $\{ e_0, \ldots, e_n \}$. The \emph{unit simplex} $\units\subset \R^n$ is the convex
  hull of the points $\{ 0,e_1,\ldots,e_n \}$.
\end{definition}

A general simplex may then be parametrized either over the standard simplex $\stds$ 
or the unit simplex $\units$.

\begin{definition}
  \label{def:simplex}
	A $n$-dimensional Euclidean simplex $s$ is the
	convex hull of $n+1$ points $p_0,\dots,p_n \in \R^m$.
	We define the barycentric map 
	\[
	  x: \stds \to s \ , \quad
	  x(\lambda) = \sum_{i=0}^n \lambda^i p_i \ ,
	\]
	which gives a parametrization of $s$ over the standard simplex $\stds$.
	We may also use a parametrization
	
	\[ 
	  y: \units \to s \ , \quad
	  y (u) = Au + p_0 \ ,
	\]
	 where $A$ is the matrix with columns $p_i - p_0$. 
\end{definition}
The mapping $y$ is usually called ``mapping onto the reference element,'' see, e.g.,~\cite[Thm. 4.4.4]{Brenner08}.
	
%
%

\subsection{Riemannian metrics on Euclidean simplices}

	A Euclidean simplex $s$ inherits a Riemannian metric from its ambient space. We may 
	then pull back this metric to either $\stds$ or $\units$. 
%

First, consider the case of $\stds$. 
	Let $v,w \in T_\lambda \stds$. Then the pullback of the Riemannian metric from $s$ onto $\stds$
	is given by
	\[
		\left\langle\sum_{i=0}^n v^i p_i, \sum_{j=0}^n w^j p_j\right\rangle_{\R^{m}}
		= \sum_{i,j=0}^n E_{ij} v^i w^j
  \]
  with
  \[
    E_{ij} = - \smallfrac 1 2 \|p_i - p_j\|_{\R^m}^2
	\]
	for $i,j = 0,\dots,n$ since $\sum{v^i}=0$ and $\sum{w^i}=0$. (See \cite[Thm. 1.2.9]{Fiedler11}.) 
	The metric is not
	determined in direction $e := (1,\dots,1) \perp_{\R^{n+1}} T_\lambda\stds$,
	so every $\tilde E_{ij} = E_{ij} + \rho$, $\rho \in \R$,
	gives the same metric on $\stds$. 
%
	Since $s$ is a Euclidean simplex, the symmetric matrix of negative squared edge lengths, $(E_{ij})$, yields a positive definite bilinear form when restricted to $T_\lambda\stds$. Conversely, if a system of prescribed ``edge lengths'' $\bar \ell_{ij}$
	is given such that the matrix with entries $-\bar \ell_{ij}^2$ yields a positive 
	definite bilinear from on $T_\lambda \stds$, then there exists a Euclidean simplex 
	with edge lengths $\bar \ell_{ij}$ (see~\cite[thm. 1.2.4]{Fiedler11}). The Riemannian 
	metric determined by $E$ will be generally referred to as $g^e$.

	Now consider the case of $\units$. The pullback of the Euclidean metric on $s$ to the unit simplex $\units$ is given by the bilinear form
	$g^\eucl_{ij} = {\sprod{p_i - p_0}{p_j - p_0}}_{\R^m}$,
	$i,j = 1,\dots,n$. Hence 
	\begin{align}\label{eq:gFromE}
		g^\eucl_{ij} =  E_{ij} - E_{0i} - E_{0j} \ ,
	\end{align}
	and the existence criterion for a Euclidean simplex with prescribed edge lengths
	is equivalent to positive
	definiteness of $g^\eucl_{ij}$ over $\R^n$, see
	\cite[Thm. 1.2.7]{Fiedler11} or \cite{Dekster87}.
	Note that~\eqref{eq:gFromE} is not true for $\tilde E_{ij}$
	as above, but relies on the specific choice that has
	been taken for $E_{ij}$ in the direction orthogonal to $T_\lambda \stds$.
	Since $\det g^\eucl_{ij} = ({n!} \vol s)^2$,
	a $(\theta,h)$-full Euclidean simplex satisfies
	$\det g^\eucl_{ij} \geq \theta^2 h^{2n}$.
	We can use this fact to estimate the eigenvalues of 
	$g^\eucl$.

\begin{lemma}
	\label{lem:eigenvaluesOfFirstFF}
	\sloppypar
	Let $p_0, \dots, p_n \in \R^m$ be the vertices of a $(\theta,h)$-full Euclidean 
	$n$-simplex, and let $g^\eucl_{ij} = {\sprod{p_i - p_0}{p_j - p_0}}_{\R^m}$
	denote the pull-back of its metric to the unit simplex $D$.
	Then the eigenvalues $\lambda_k$ of $g^\eucl$ satisfy
	\[
		\theta h n^{1-n}
		\leq \,\,
		\sqrt{\lambda_k}
		\,\, \leq \,\,
		h n.
	\]
\end{lemma}
\begin{proof}
	Using the matrix $A$ from Definition~\ref{def:simplex}, notice that $g^\eucl = A^t A$. Hence the eigenvalues of $g^\eucl$ are the
	squared singular values of the matrix $A$.
	For any $n$-simplex $s$, the radius $r$ of its insphere satisfies $\vol_n(s) = \frac r n \vol_{n-1}(\bdry s)$.
	As $D$ has volume $\frac 1 {n!}$ and $\bdry D$ has volume
	$\frac {n + \sqrt n}{(n-1)!}$, this gives
	\[
		r = \frac 1 {n + \sqrt n} \geq \frac 1 {2n} \ .
	\]
	This means that any vector $v \in TD$ with
	length $\frac 1 n \leq 2r$ can be represented as $p-q$
	with points $p,q \in D$. Its image in $s$ is
	$Ap - Aq$, which must be shorter
	than the diameter of $s$. Since the diameter
	of a Euclidean simplex is the length of its longest
	edge, it follows that $\norm A \leq nh$, which implies that $\lambda_{\max}
	\leq (nh)^2$. On the other hand,
	\[
		\lambda_{\min} (nh)^{2n-2}
		\geq \lambda_{\min} \lambda_{\max}^{n-1} \geq \det g^\eucl
		\geq \theta^2 h^{2n} \ ,
	\]
	which proves the claim.
\end{proof}

\begin{remark}
  One can also express the volume of $s$ using the matrix $E$ representing the metric of $s$ parametrized over the 
  standard simplex $\stds$. Let $e=(1,\dots,1)\in R^{n+1}$. The volume of $s$ is $\smallfrac 2{n!}(-\det M_+)^\half$,
	where
	\[
		M_+ = \begin{pmatrix} 0 & -\smallfrac 1 2 e^t \\
	    	                  -\smallfrac 1 2 e & E
	      	\end{pmatrix} \in \R^{(n+2)\times(n+2)} 
	\]
	is $-\frac 1 2$ times the \emph{Cayley--Menger} matrix
	\cite[Rem. 1.4.4]{Fiedler11}
	(a more classical reference is \cite[Thm. 40.1]{Blumenthal52}).
\end{remark}

The previous result allows the following estimate.
\begin{proposition} \label{prop:innerprod est}
	Let $v \in T_\lambda \stds$ and let $Y_0,\ldots,Y_n$ be vectors at a point on a Riemannian manifold 
	(M,g). If $g^e$ is a Euclidean metric on $\stds$ making $\stds$ a $(\theta,h)$-full simplex, 
	then we have the following estimate:
	\[
		\left|\sum_{i=0}^n v^iY_i \right|_g \leq \frac{n^n}{\theta h}\absval[g^e]v \sum_{i=0}^n \absval[g]{Y_i}.
	\]
\end{proposition}
\begin{proof}
	The key fact is that we can pull back to the unit simplex $D$ with pulled back metric $g^\eucl$ and estimate
	with Lemma \ref{lem:eigenvaluesOfFirstFF}:
	\begin{align*}
	  \left| \sum_{i=0}^n v^i Y_i\right|_g &=\left| \sum_{i=1}^n v^i (Y_i-Y_0)\right|_{g} \\
	  &\leq \sqrt{\sum_{i=1}^n\absval{v^i}^2}\sqrt{\sum_{i=1}^n\absval[g]{Y_i-Y_0}^2}\\
	  &\leq \frac{n^{n}}{\theta h}\absval[g^e]v \sum_{i=0}^n\absval[g]{Y_i}.
	\end{align*}
\end{proof}


As a straightforward consequence of the polarization identity, estimates for symmetric bilinear forms can be derived 
from estimates of the associated quadratic forms:
\begin{lemma}
\label{lem:fromNormToScalarProduct}
Let $T$ be a symmetric bilinear form on a vector space $V$, and let $g$ be an inner
product on $V$ (i.e., a symmetric and positive definite bilinear form). Suppose that $|T(v,v)| \leq C |v|_g^2$ for all vectors $v$ in $V$. Then
for all vectors $v,w\in V$, $\absval{T(v,w)} \leq C \absval[g]v \absval[g]w$.
\end{lemma}
The following is a standard estimate. 
\begin{lemma}
		\label{lem:metriccompare}
	Let $g$ and $\bar g$ be inner products on $\R^n$ 
	such that all eigenvalues of $g$ (with respect to the Euclidean inner product) are larger than $\lambda_{\min} > 0$
	and $\absval{g_{ij} - \bar g_{ij}} \leq \eps n^{-1} \lambda_{\min}$.
	Then
	$\absval{(g - \bar g)$ $(v,v)} \leq \eps \absval[g]v^2$.
\end{lemma}

%
%
%
%
%
%
%
%
%

We now apply this lemma to Euclidean metrics arising from simplices.
\begin{proposition}
	\label{prop:almostEqualMetricsFromEdgeLengths}
	There is a constant $\alpha = \alpha(n)$ with the
	following property:
	If $\ell_{ij}$ are the edge lengths of a Euclidean
	$n$-simplex that defines a $(\theta,h)$-full metric $g$
	on $\stds$, and $\bar \ell_{ij}$ define
	a second system of lengths (not a priori assumed to define a simplex) with
	$\absval{\ell_{ij} - \bar \ell_{ij}}
	\leq \alpha \eps \theta^2 \ell_{ij}$,
	where $\eps < \frac 1 2$,
	then there is a Euclidean $n$-simplex $\bar s$
	with edge lengths $\bar \ell_{ij}$,
	and its  Euclidean metric $\bar g$ over $\stds$
	satisfies $\absval{(g - \bar g)(v,v)}
	\leq \eps \absval[g] v^2 $ for every
	$v \in T\stds$.
\end{proposition}
\begin{proof}
	Let $x,y>0$ be real, such that
	$\absval{x-y} \leq \delta x$, for some $\delta < 1$, then
	$\absval{x^2 - y^2} \leq 3 \delta x^2$. 
	By assumption and since $\theta \leq 1$, we hence have
	$\absval{\ell_{ij}^2 - \bar \ell_{ij}^2} \leq
	3 \eps \alpha \theta^2 \ell_{ij}^2$ whenever $\alpha \leq 2$. Therefore,
	$\absval{E_{ij} - \bar E_{ij}} \leq \frac 3 2 \eps \alpha \theta^2 h^2$.
	Working over the unit simplex and using~\eqref{eq:gFromE} to define 
	$g^\eucl$ and $\bar g^\eucl$ we obtain
	that $\absval{g^\eucl_{ij} - \bar g^\eucl_{ij}} \leq \frac 9 2 \eps \alpha \theta^2 h^2$.
	
	Let $\lambda_{\min}$ be the smallest eigenvalue of $g^\eucl$. 
	If $\alpha$ is chosen so small that $\frac 9 2 \alpha n  \leq n^{2-2n}$, 
	then by Lemma \ref{lem:eigenvaluesOfFirstFF} we have that $\frac 9 2 \alpha n \theta^2 h^2
	\leq \lambda_{\min}$. By Lemma \ref{lem:metriccompare},
	we get $\absval{(g - \bar g)(v,v)} \leq
	\eps \absval[g] v^2$. In particular,
	$\bar g$ is positive definite since $\eps < \frac 1 2$.
%
\end{proof}

\subsection{Euclidean simplices from geodesic simplices}

Let $(M,g)$ be  a complete Riemannian manifold, and let $p\in M$. Then Proposition~\ref{prop:almostEqualMetricsFromEdgeLengths} 
can be used to compare the flat metric on a Euclidean simplex whose edge lengths are defined by geodesic
distances $\dist(\exp_p v, \exp_p w)$ to the flat metric on the Euclidean simplex whose edge lengths are defined by distances $\absval[g_p]{v-w}$ in the tangent space at $p$. 
In order to use the conclusions of  Proposition~\ref{prop:almostEqualMetricsFromEdgeLengths}, 
we require estimates on the distortion induced by the exponential map. 
In this section, we provide these estimates.

	
Before discussing distortion induced by the exponential map, we compare
lengths of geodesics for different metrics. 

\begin{proposition}
\label{prop:geodesiccompare}
  Let  $g$ and $g'$ be Riemannian metrics on an open set $U\subset \R^n$ and let 
  $\gamma$ and $\gamma'$ be minimizing geodesics for $g$ and $g'$ contained entirely in $U$
  such that $\gamma(0) = \gamma'(0)$ and $\gamma(1)=\gamma'(1)$. If there exists $0\leq \alpha<1$ such that 
  \[
    \big|\absval[g]{v} - \absval[g']{v}\big| \leq	\alpha \absval[g]{v}
  \]
  for all $v\in TU$, then the lengths of $\gamma$ and $\gamma'$ satisfy
  \[
    | l_g(\gamma) - l_{g'}(\gamma') | \leq \alpha l_g(\gamma) \ .
  \] 
\end{proposition}

\begin{proof}
  We first note that 
  \[
    (1-\alpha)\absval[g]{v} \leq 
       \absval[g']v.
  \]
  Since $\gamma$ is minimizing for $g$, we have
  \begin{align*}
	  (1-\alpha) l_{g'}(\gamma') 
        &= (1-\alpha)\int_0^1 \absval[g]{\dot\gamma'}dt 
        + (1-\alpha)\int_0^1 (\absval[g']{\dot\gamma'} - \absval[g]{\dot\gamma'})dt \\
        &\geq  (1-\alpha)l_{g}(\gamma')-\alpha(1-\alpha) \int_0^1 \absval[g]{\dot\gamma'} dt\\
        & \geq (1-\alpha)l_{g}(\gamma)-\alpha l_{g'}(\gamma').
	\end{align*}  
	It follows that  
	\[
	  l_{g}(\gamma) - l_{g'}(\gamma') \leq \alpha l_g(\gamma).
	\]
	
	Similarly, since $\gamma'$ is minimizing for $g'$, we have
	\begin{align*}
	  l_{g}(\gamma) 
        &= \int_0^1 \absval[g']{\dot\gamma}dt + \int_0^1 (\absval[g]{\dot\gamma} - \absval[g']{\dot\gamma})dt \\
        &\geq  l_{g'}(\gamma)-\alpha l_g(\gamma)\\
        & \geq l_{g'}(\gamma')-\alpha l_g(\gamma),
	\end{align*}  
	so 
	\[
	  l_{g'}(\gamma') - l_{g}(\gamma) \leq \alpha l_g(\gamma).
	\]
\end{proof}	
	
The main comparison result of this section is the following.
\begin{proposition}	 \label{prop:compareElls}
	Let $(M,g)$ be a complete Riemannian manifold. Let $p\in M$, and let $U\subset T_p M$ 
	be a ball of radius $r$ centered at $0\in T_p M$ such that $\exp_p(U)$ is geodesically convex. Pull back the metric $g$ to $U$ via $\exp_p$. 

	Suppose $c:[0,\tau] \to U$ is a geodesic (with respect to the pulled back metric) and
	$b:[0,\tau] \to U$ is a straight line with same endpoints
	$b(0) = c(0)$ and $b(\tau) = c(\tau)$. There are constants
	$\alpha = \alpha(C_0,n)$ and
	$\beta=\beta(n)$ such that if $r < \alpha$, then the
	$g$-lengths of $c$ and the Euclidean length of $b$ satisfy
	\[
		\absval{l_g( c) - l_\eucl(b)} \leq \beta C_0 r^2 l_g( c) \ .
	\]
\end{proposition}
\begin{proof}
	For a geodesic ball with radius $r<1$, there is a constant $\beta=\beta(n)$ with
	\[
		\bigabsval{\absval[g]v - \absval[\eucl] v} \leq \beta C_{0}r^2 \absval[g] v \ ,
	\]
	for all $v\in TU$ \cite{Kaul76}. If we suppose $r<\alpha$ with $\alpha<1$ so small
	that 
	\[
	  \beta C_0\alpha^2 <1\ ,
	\]
  then the proposition follows from Proposition \ref{prop:geodesiccompare}.
\end{proof}
The following is the main result we will use in Section \ref{sec:Approximation}.
\begin{corollary} \label{cor:metrics from points est}
		Let $(M,g)$ be a complete Riemannian manifold. Let $a\in M$
		and suppose that $p_0,\ldots, p_k$ are points in a convex
		geodesic ball centered at $a$ such that the lengths $\ell_{ij}=\dist(p_i,p_j)$
		form a Euclidean simplex and let $g^e$ denote the induced metric on $\stds$. Let $X_i\in T_aM$ be defined such that $\exp_a X_i = p_i$
		and let $\sigma:\stds \to T_aM$ denote the barycentric coordinates to the simplex determined by the $X_i$. 
		
		There are $\alpha = \alpha(n)$ and $\beta=\beta(n)$ such that if $g^e$ is a  $(\theta,h)$-full metric
		with $h < \alpha$ then 
		\[
		\big|\absval[g^e]v^2 - \absval[g]{\sigma(v)}^2\big|
			\leq \beta C_0 \theta^{-2}h^2 \absval[g^e] v^2 
		\] 
		for every $v \in T\stds$.
		It follows that there is an $\alpha' = \alpha'(n,C_0,\theta)$ so that if furthermore $h<\alpha'$, then 
		\[
		\big|\absval[g^e]v^2 - \absval[g]{\sigma(v)}^2\big|
			\leq 2\beta C_0 \theta^{-2}h^2 \absval[g]{\sigma(v)}^2.
		\]
\end{corollary}
Note that $\absval[g]{\sigma(v)}$ is equal to the length of the vector $v$ measured in the Euclidean metric $g^e$ on $\stds$ determined by the lengths $\bar{\ell}_{ij}=|X_i-X_j|$.
\begin{proof}
	The fact that there is a $\gamma = \gamma(n)$ such that for each corresponding side,
		\[
 		\absval{\ell_{ij} - \bar{\ell}_{ij}} \leq 	 \gamma C_{0} r^2 \ell_{ij}
		\]
	follows from Proposition \ref{prop:compareElls}.
	Taking 
	\[
		\eps = \frac\gamma\alpha C_0 \theta^{-2} h^2 
	\]
	in Proposition \ref{prop:almostEqualMetricsFromEdgeLengths} (using the $\alpha$ from that proposition) we get that
		\[
		\big|\absval[g^e]v^2 - \absval[g]{\sigma(v)}^2\big|
			\leq \frac\gamma\alpha C_0 \theta^{-2} h^2\absval[g^e] v^2 .
		\] 
	The last statement follows from the fact that 
		\[
		 \absval[g^e] v^2 \leq \big|\absval[g^e]v^2 - \absval[g]{\sigma(v)}^2\big| + \absval[g]{\sigma(v)}^2.
		\]
	so we need to take $\alpha'$ small enough so that $\beta C_0 \theta^{-2}(\alpha')^2<1/2$.
\end{proof}
One last fact that will be important in the sequel.
\begin{proposition} \label{prop:inj radius}
  In Proposition \ref{prop:compareElls} and Corollary \ref{cor:metrics from points est}, we may replace the convexity assumption by making the constants $\alpha$ depend also on the injectivity radius $\iota$.
\end{proposition} 
\begin{proof}
  This is the well-known fact that every point has a convex geodesic ball around it, and the radius of that ball can be estimated by the curvature bound $C_0$ and the injectivity radius $\iota$.
\end{proof}


\section{Barycentric coordinates from center of mass}
\label{sec:defBarycCoords}


\subsection{Definition of the barycentric coordinate map}

Let $(M,g)$ be a complete $m$-dimensional Riemannian manifold, $m > 1$, and $\stds$ be the
	$n$-dimen\-sional standard simplex $\{\lambda \in \R^{n+1} \with \lambda^i \geq 0,
	\sum \lambda^i = 1\}$. For points $p_0,\dots,p_n \in M$, consider the function
	\[
		E: M \times \stds \to \R
	\]
	given by 
	\[
		E(a,\lambda) = \lambda^0 \dist^2(a,p_0) + \dots + \lambda^n \dist^2(a,p_n).
	\]
	
The minimizer of $E(\argdot, \lambda)$, if it exists, is called the \textbf{center of mass} for the measure
$\sum \lambda^i \delta_{p_i}$, where $\delta_p$ is the Dirac delta point mass at the point $p$.
  It is expecially useful to see that the center of mass can be formulated
as the zero of a function.

\begin{proposition}[\cite{Karcher77}]
  \label{prop:criticalptsofE}
	Local minimizers of $E(\argdot, \lambda)$ for fixed $\lambda$
	are zeroes of the section $F: M \times \stds \to TM$ given by
  \begin{equation}
  \label{eq:Fdef}
		F(a,\lambda) = \lambda^0 X_0|_a + \dots + \lambda^n X_n|_a,
		\qquad \text{where} \quad
	  X_i = \frac 1 2 \grad \dist^2(\argdot, p_i).
	\end{equation}
\end{proposition}

  Given that the center of mass is just a solution to $F=0$, where $F$ is given by Equation \ref{eq:Fdef}, 
  one can apply the implicit function theorem to get a solution. One caution is that in order to
  apply the implicit function theorem, one needs to trivialize the bundle $TM$, and hence we must
  choose a connection to do the trivialization. It is natural to use  
  the Riemannian connection. In later computations, we will see the connection appear
  in formulas for differentials of mappings (not just derivatives of forms and vector fields),
   and this is because the definition of the mapping $F$ requires this trivialization.

The existence of the center of mass is hard to prove in general,
cf. \cite{Sander13}. For sufficiently small balls, however, there is a unique minimizer. The 
following proposition is a generalization of one in \cite{Karcher77} due to Kendall \cite{Kendall90}
(see also \cite{Afsari11}).
\begin{proposition}
	If the points $p_i$ lie in a ball whose radius is less than half the convexity radius,
	then $E(\argdot, \lambda)$ has a unique minimizer.
\end{proposition}

Since we will always want to consider unique centers of mass, we make the following assumption.
\begin{assumption}
	From the rest of this paper, we only consider $p_0,\dots,p_n$ that
	lie in a 
	ball $B$ whose radius is less than half the convexity radius.
	Note that, as in Proposition \ref{prop:inj radius},
	if we choose ball to have small enough radius, where ``small 
	enough'' depends on the curvature bound $C_0$ and the injectivity radius
	$\iota$, this is true (the original bound in \cite{Karcher77} is of this form). 
	Hence this extra assumption is superfluous if 
	we require our future antecedents of the form $h<\alpha$ to have $\alpha$ depend on 
	$\iota$ as well as $C_0$.	
\end{assumption}

We are now able to define the barycentric mapping.
\begin{definition}
	For a given $\lambda \in \stds$, let $x(\lambda)$ be the
	minimizer of $E(\argdot,\lambda)$ in $B$. We call $x$
	the \textbf{barycentric mapping} with
	respect to vertices $p_1,\dots,p_n$. Its image in $M$
	is called the corresponding \textbf{Karcher simplex}.
\end{definition}
\begin{remark}
	\begin{subenum}
	\item In case $M$ is the Euclidean space, $x$ is just the canonical
	parametrisation $\lambda \mapsto \sum \lambda^i p_i$, because
	$\dist^2(p,q) = \absval{q-p}^2$ and $X_i|_q = q - p_i$.

	\item	
	\label{rem:properSubsimplices}
	For $\lambda^i = 0$, the value $x(\lambda)$ is independent of $p_i$.
	So the facets of the standard simplex are mapped to ``Karcher
	subsimplices'' which only depend on the vertices of the subsimplex.

	\item Because $x$ is continous, the Karcher subsimplices
	form the boundary of a Karcher simplex: $\bdry (x(\stds)) = x(\bdry \stds)$.
	\end{subenum}
\end{remark}

The barycentric mapping behaves well with respect to submanifolds.
\begin{proposition}[geodesic submanifolds]
	Let $e_i$ be the $i$-th Euclidean basis vector of $\R^{n+1}$.
	Then $x(t e_j + (1-t) e_i) = \gamma(t)$, where $\gamma$ is
	the unique shortest geodesic with $\gamma(0) = p_i$ and
	$\gamma(1) = p_j$.
	If all $p_i$ lie in a common totally geodesic submanifold $N \subset M$,
	then so does $x(\lambda)$ for arbitrary $\lambda \in \stds$.
\end{proposition}
\begin{proof}
	The first claim is clear.
	For the second, we will show that the center of mass restricted to 
	$N$ is also the center of mass on $M$ since $N$ is totally geodesic.
	Let $k_i := \frac 1 2 \dist^2(\argdot, p_i)$.
	As $E$ is convex on $B$, there is a unique minimizer $a$
	of $E(\lambda, \argdot)|_N$, so there are coefficients
	$\lambda^0,\dots,\lambda^n$ with
	\[
	  \lambda^0 \grad (k_0|_N) + \dots + \lambda^n \grad (k_n|_N) = 0.
	\]
	As $N$ is totally geodesic in $M$, it follows that $\grad (k_i|_N) = 
	(\grad k_i)|_N = X_i|_N$. Hence $\lambda^0 X_0 + \dots
	+ \lambda^n X_n = 0$ at the point $a \in N$. Thus $a=x(\lambda)$.
\end{proof}
\begin{remark}
	In general, $s$ is the convex hull of the $p_i$
	(the smallest convex set containing all $p_i$)
	if and only if all subsimplices of $s$ are totally
	geodesic. In fact, if a Karcher triangle
	is not totally geodesic, the geodesics connecting
	any points on the geodesic edges will not be
	contained in the triangle.
\end{remark}

\begin{bibnote}
We already remarked that
one can consider $\lambda$ as a point measure on
$M$ that assigns the mass $\lambda^i$ to the
point $p_i$. For a general probability measure $\mu$
on $M$, \cite{Karcher77} speaks of the minimizer
of
\[
	E_\mu(a) := \int_M \dist^2(a,p) \d\mu(p)
\]
as ``Riemannian center of mass'', but the
subsequent literature has mostly called it the
``Karcher mean'' with respect to the measure $\mu$,
probably initiated by Kendall in
\cite{Kendall90} (unfortunately, Karcher himself is
not very happy with the naming, see \cite{Karcher14}).
The concept seems to go back
to Cartan (see the historic overview
in \cite{Afsari11}), but had not been used by others
until the work of \cite{Grove73}.
	
Karcher himself used the center of mass to
retrace the standard mollification procedure
of Gauss kernel convolution in the case
of functions that map into a manifold.
Considering the center of mass as a function
from an interesing finite-dimensional space of measures into $M$,
as we use it, has been done
by \cite{Rustamov10}. For other recent applications, see the discussion in
sec.~\ref{sec:relatedWork}.
\end{bibnote}


\subsection{Derivatives of the barycentric coordinate map}

Since the next proposition uses pullback bundles, we take this opportunity to recall the calculus of pullback
vector bundles. Recall that given a smooth vector bundle 
\[
  \pi:E \to M
\]
and smooth map
\[
  \phi:N \to M,
\]
one can define the pullback bundle 
\[
  \phi^*\pi : \tilde{E} \to N
\]
with
\[
  \tilde{E} = \mathop{\bigsqcup}_{x \in N}E_{\phi(x)},
\]
where $\sqcup$ represents the disjoint union. (We will always be considering injective
maps, so the disjoint union may be replaced with the union.) 
For the tangent bundle of $M$, we use the notation $\phi^*TM$. 
If $V$ is a vector field on $M$, we will denote corresponding
sections $\phi^*TM \to N$ by $\tilde{V}|_p = V|_{\phi(p)}$. We will not
stress the discrepency between the section $\tilde{V}$ evaluated at $p$ and the
section $V$ evaluated at $\phi(p)$, even though the latter requires a (more general) 
vector field on $M$. If we have a smooth map $\phi$, the differential $d\phi$ is a section of the bundle 
\[
  T^*N \otimes \phi^* TM \to N
\]
(the dual-space star in $T^*N$ is not to be confused with the pull-back
star in $\phi^* TM$).
Given Riemannian metrics $g^N$ and $g^M$ on $N$ and $M$ with Riemannian connections $\nabla^N$ and $\nabla^M$, the Hessian $\nabla d \phi$
is a section of the bundle
\[
  T^*N \otimes T^*N \otimes \phi^* TM \to N,
\]
given by differentiating the section $d\phi$ using the induced connection on the bundle $ T^*N \otimes \phi^* TM \to N$.

In particular, considering the map $x:(\stds, g^e) \to (M,g)$, we get the following derivatives:
\begin{subenum}
	\item The section $dx$ of the bundle $T^*\stds \otimes x^*TM \to \stds$. We denote $V=dx(v)$.
	\item The section $\nabla dx$ of the bundle $T^*\stds \otimes T^*\stds \otimes x^*TM \to \stds$ given by 
	  \begin{align}
	    \label{eqn:nabladx}
	    \nabla dx (v,w) &= \nabla_{dx(v)} dx(w) - dx(\nabla_v w)
	        = \nabla_V W - dx(\nabla_v w),
	  \end{align}
	  where the first connection on the right is the Riemannian connection for $g$ and the second connection is the Riemannian connection for $g^e$.
\end{subenum}

The following is a direct consequence of Proposition \ref{prop:criticalptsofE}.
\begin{proposition}
  \label{prop:GequalFcomposex}
  The map 
  \[
    G: \stds \to x^*TM
  \]
  given by 
  \begin{align*}
      G(\lambda) &= F(x(\lambda), \lambda)\\
                 & = \sum{\lambda^i X_i|_{x(\lambda)}}   
  \end{align*}
  is the zero section.
\end{proposition} 

It follows that $dG$ is the zero section
  of the bundle $T^*\stds \otimes x^*TM \to \stds$, and hence
	\begin{equation}
	  \label{eqn:dGIsZero}
	  0= dG(v)= \sum{v^i X_i}+\sum{\lambda^i \nabla_{dx(v)}X_i}.
	\end{equation}

We will use this and also its derivative to calculate $dx$ and $\nabla dx$. 
It will be useful to name the two parts to the above formula.

\begin{definition}[cf. Corollary \ref{cor:metrics from points est}]
  \label{def:sigma}
  Define the section $\sigma$ of the bundle $T^*\stds \otimes x^*TM \to \stds$ to be
  \[
    \sigma(v) =  -\sum{v^i X_i}
  \]
  and define the section $A$ of the bundle $(x^*TM)^* \otimes x^*TM \to \stds$ to be 
  \[
    A(V) = \sum{\lambda^i \nabla_{V}X_i},
  \] 
  where each of these sections is evaluated at $\lambda \in \stds$.
\end{definition}

These are defined so that 
\[
  dG(v) = -\sigma(v) + A(dx(v)).
\]

The relationship of $\sigma$ and $A$ with $dx$ and $\nabla d x$ is given by the following.

\begin{proposition}
	\label{prop:derivativesOfF}
  Let $v$ and $w$ be vector fields on $\stds$ and let $V=dx(v)$ and
  $W=dx(w)$ be the corresponding sections of $x^*TM\to \stds$.	
	The differential $dx$ and Hessian $\nabla d x$ satisfy 
	\begin{align*}
	  A(dx(v)) &= \sigma(v),
	\end{align*}
	and
	\begin{align*}
		A(\nabla d x(v,w)) &=
				- \left(
				  \sum{w^i \nabla_{V}X_i} \right.
				  +\sum{v^i \nabla_{W}X_i}
				  \left. + \sum{\lambda^i \nabla^2_{W,V} X_i} 
				\right),
	\end{align*}
	evaluated at a point $\lambda \in \stds$.
\end{proposition}

\begin{proof}
	These will both follow from differentiating the identity from Proposition \ref{prop:GequalFcomposex}; the first claim is already treated by 
	(\ref{eqn:dGIsZero}).
	We take the next derivative using the covariant derivative of 
	  $dG$, which is a section of $T^*\stds \otimes x^*TM \to \stds$, using the connection $\nabla^{e,g}$ determined by
	  the Riemannian connections of $g^e$ and $g$. To simplify the writing, we use $\nabla$ for each connection, 
	  with the understanding that the context will make it clear whether the connection needs to be
	  computed with respect to $\nabla^{e,g}$, $\nabla^e$, or $\nabla^g$ (which is primarily used on the pullback
	  connection as $x^*TM \to \stds$). 
	\begin{align*}
	  \nabla dG(w,v) &= \nabla_w dG(v) - dG(\nabla_w v)\\
	  										&= \sum{w(v^i) X_i}+ \sum{v^i \nabla_{W} X_i}
	  											+\sum{w^i \nabla_{V}X_i}\\
	  										&\qquad  	+\sum{\lambda^i \nabla_{W}\nabla_{V}X_i} 
	  										  - \sum{(\nabla_w v)^i X_i}-\sum{\lambda^i \nabla_{dx(\nabla_w v)}X_i}\\
	  										&=  \sum{v^i \nabla_{W} X_i}
	  											+\sum{w^i \nabla_{V}X_i}
	  											+\sum{\lambda^i \nabla^2_{W,V}X_i}\\
	  										&\qquad	+\sum{\lambda^i \nabla_{\nabla_W V}X_i}-\sum{\lambda^i \nabla_{dx(\nabla_w v)}X_i}\\
	  										&= \sum{v^i \nabla_{W} X_i}
	  											+\sum{w^i \nabla_{V}X_i}\\
	  											&\qquad +\sum{\lambda^i \nabla^2_{W,V}X_i}
	  											+\sum{\lambda^i \nabla_{\nabla dx (w,v)}X_i}
	\end{align*}
	for vector fields $v$ and $w$ on $\stds$ evaluated at $\lambda$. The last equality follows from the definition of 
	$\nabla dx$ (see Equation \ref{eqn:nabladx}).
\end{proof}


\section{Estimates} \label{sec:Approximation}


In Proposition
\ref{prop:almostEqualMetricsFromEdgeLengths} from Section
\ref{sec:flatMetrics} we saw that
if two flat Riemannian
metrics on $\stds$ assign almost the same
lengths to the edges of the standard simplex,
the metrics are almost equal.
In Proposition
\ref{prop:compareElls} we saw that
in normal coordinates, geodesic distances and
Euclidean distances of the coordinates are
almost the same. These will give us the flexibility to 
give our estimates with respect to the appropriate metrics.

In order to 
compare $x^*g$ to $g^e$, we add an
intermediate step by using $\sigma$ from Definition \ref{def:sigma}. 
For a fixed $\lambda$, consider 
the Euclidean simplex $\bar s$ in the tangent space
$T_{x(\lambda)} M$ with vertices
$X_i|_{x(\lambda)} = (\exp_{x(\lambda)})\inv p_i$.
The (classical) barycentric coordinate
map for this simplex is, in fact, $\sigma$.
The simplex $\bar s$ inherits the metric
$g_{x(\lambda)}$ from its ambient space
$T_{x(\lambda)} M$. We show that $dx - \sigma$
is small, so that $x^*g$ and $\sigma^*g$
almost agree. And by
Proposition \ref{prop:compareElls},
$\absval{X_i - X_j}$ is almost $\dist(p_i,p_j)$,
hence Proposition
\ref{prop:almostEqualMetricsFromEdgeLengths}
compares the flat metrics $\sigma^*g$ and $g^e$. The relevant
estimate is Corollary \ref{cor:metrics from points est}.
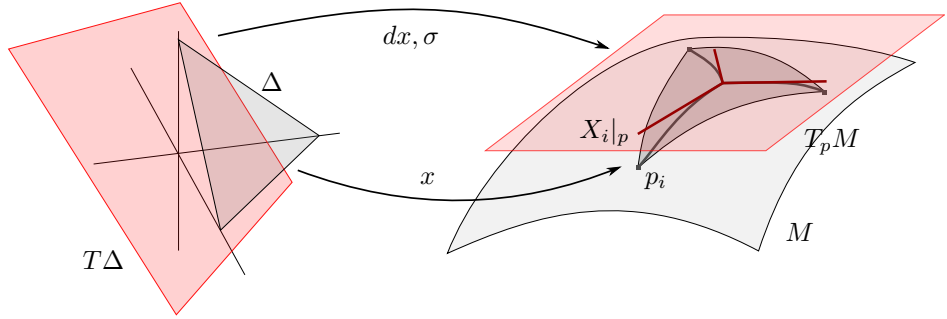
\begin{figure}
	\caption{Notation overview for mappings: $x$ maps the
	standard simplex $\stds$ into the grey Karcher
	simplex $s$, while $\sigma$ maps $\stds$ onto the
	convex hull $\bar s$ of the $X_i|_p$.}
	\vspace{3ex}
	\begin{overpic}[scale=0.4]{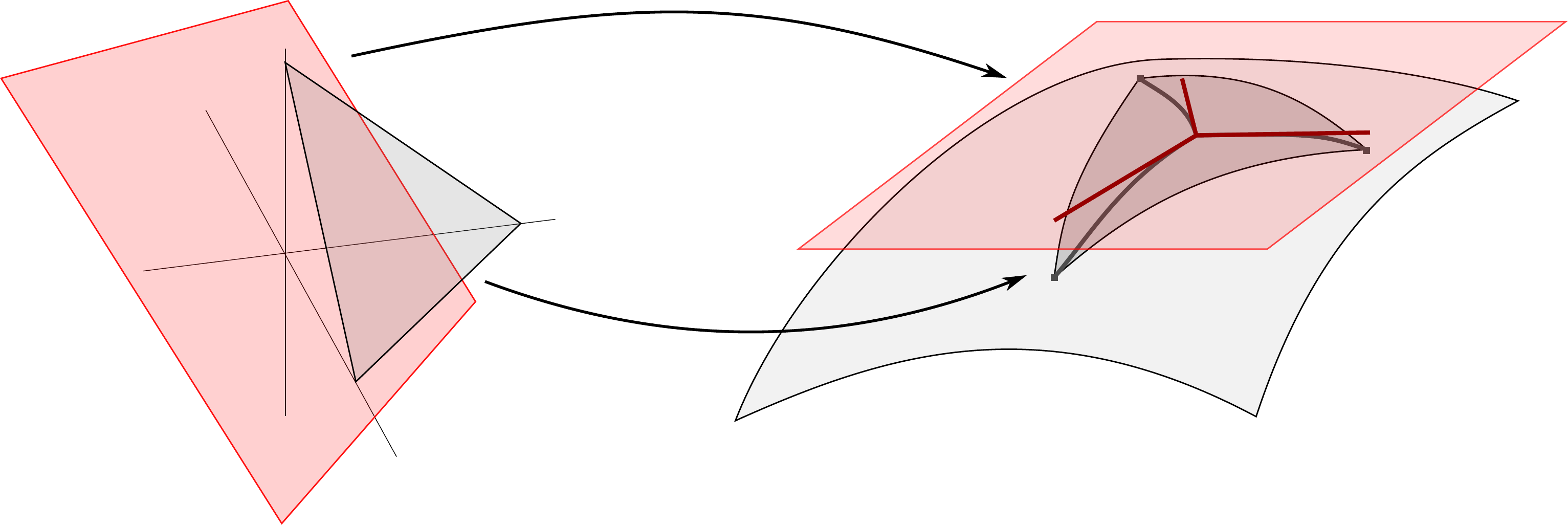}
	\put(8,5){$T\stds$}
	\put(27,24){$\stds$}
	\put(44,14){$x$}
	\put(40,29){$dx,\sigma$}
	\put(83,8){$M$}
	\put(85,18){$T_pM$}
	\put(68,14){$p_i$}
	\put(61,19){$X_i|_p$}
	\end{overpic}
\end{figure}

We now introduce a notation we will use for estimates.
\begin{notation}
	By $a \simleq b$, we mean that there is a constant $\alpha = \alpha(n)$
	with $a \leq \alpha b$. We decided not to include the
	geometric properties in this suppressed constant to make
	their influences clear.
	In the following, we will give estimates of the form 
	\[
	  \absval {A-B} \leq \beta \absval{B}.
	\]
	We note that for sufficiently small $\beta$, this estimate implies the estimates
	\[
	  (1-\beta)\absval B \leq \absval A \leq (1+\beta) \absval B
	\]
	since the triangle inequality gives
	\[
	  \absval B - \absval{A-B} \leq \absval A \leq \absval B + \absval{A-B}.
	\]
\end{notation}


\subsection{Preliminary estimates}

From Proposition \ref{prop:derivativesOfF}, we see that if we can show that $A$ is invertible and we 
can bound its operator norm, we can get estimates on $dx$ and
$\nabla dx$. First, we will need an estimate 
for derivatives of the distance function. Recall the notation from Section \ref{sec:main} before 
Theorem \ref{thm:main1}. The Jacobi field estimate from Corollary
\ref{corol:estimateForNablaX} in the Appendix shows the following.
\begin{proposition}
	\label{prop:estimateForNablaX}
	There exists $\alpha=\alpha(C_0)$ such that if $h<\alpha$ and $p_0, \ldots, p_n$ are 
	in a convex ball $B$ of diameter less than $h$, then	
	for all $i=0,\ldots,n$ and for $V \in T_q M$, where $q \in B$, we have
	\[
		\absval{\nabla_V X_i - V} \simleq C_0 h^2 \absval{V} \label{eqn:first deriv est}
	\]
	and  
	\[
	  \absval{\nabla^2_{V,V} X_i} \simleq C_{0,1} h \, \absval{V}^2. \label{eqn:second deriv est}
	\]
\end{proposition}

We may now estimate $A$ as an operator. 
In the next proposition, we use $\norm{T}$ to denote the operator norm of the map $T$ considered with the metric $x^*g$.
\begin{proposition} 
\label{prop:AandAinverse}
  There exists $\alpha = \alpha(n,C_0)$ such that if $h<\alpha$, the section $A$ of the bundle $(x^*TM)^*\otimes x^*TM \to \stds$ is pointwise invertible, i.e., there is another section $A\inv$ of the same bundle such that such that for each section
  $v$ of the bundle $x^*TM \to \stds$, $A^{-1}(A(v)) = A(A^{-1}(v)) = v$. Furthermore, 
  there exists $\beta = \beta(n)$ such that $A$ satisfies
  the following estimates:
  \begin{align*}
    \norm{A-\id} &\simleq C_0 h^2
  \end{align*}
  and
  \begin{align*}
    \absval[g]{V} 
		    &\simleq \absval[g]{A(V)}, \\
    \absval[g]{A(V) - V} 
		    &\simleq C_0 h^2 \absval[g]{A(V)}
  \end{align*}
  for any section $V$ of $x^*TM$.
\end{proposition}

\begin{proof}
	By Proposition \ref{prop:estimateForNablaX}, one has 
	\[
	  \absval{\nabla_V X_i - V} \simleq C_0 h^2 \absval V
	\]
	for all tangent vectors $V$, or, in terms
	of operator norms, 
	\[
	  \norm{\nabla X_i - \id} \simleq C_0 h^2.
	\]
	As all $\lambda^i$ are positive and sum up to one, this estimate
	carries over to 
	\begin{equation}
  \label{eq:normest}
	  \norm{A - \id}\simleq C_0 h^2. 
	\end{equation} 
	It follows that there exists $\alpha = \alpha(C_0)$ such that
	if $h<\alpha$ then $A$ is invertible.
	
	For the second estimate, we see that 
	\begin{align*}
	  \absval[g]{V}&=\absval[g]{(\id - A)V+A(V)}
	     \leq \norm{\id-A}\absval[g]{V} + \absval[g]{A(V)}.
	\end{align*}
	Using Equation \ref{eq:normest}, we get 
	\[
	  (1-\beta C_0 h^2) \absval[g]{V} \leq \absval[g]{A(V)}
	\]
	for some function $\beta =\beta(n)$ and so if we make $\alpha$ small enough so $\beta C_0 \alpha^2 <\frac 1 2$ 
	then, we have
	\[
	  \absval[g]{V} \simleq  \absval[g]{A(V)}.
	\]
	
	Finally, we get
	\[
	  \absval[g]{A(V) - V} \simleq \norm{A-\id} \absval[g]{V}\simleq \norm{A-\id} \absval[g]{A(V)}.
	\]
\end{proof}

For the rest of this section, let 
$p_0,\dots,p_n$ be distinct points inside a convex
	ball of radius $h$. Note that this is not a big assumption since we can
	just adjust out constants to depend on the injectivity radius as well,
	as in Proposition \ref{prop:inj radius}.

\subsection{First derivative estimates}

We can now estimate $dx$. 
\begin{theorem}
\label{thm:EstimateFordx}
	There exists $\alpha = \alpha(n,C_0)$ and $\alpha' = \alpha'(n,C_0,\theta)$ such that
 if $h < \alpha$ 
	then
	\[
		\bigabsval{dx(v) - \sigma(v)} \simleq
				C_0 h^2 \absval{\sigma(v)}
	\]
	and if furthermore $g^e$ is the metric of a $(\theta,h)$-full simplex
	 	with $h<\alpha'$, 
	then
	\[
	  \bigabsval{dx(v) - \sigma(v)} \simleq
				C_0 h^2 \absval[g^e]{v},
	\]
  for	any tangent vector $v \in T_\lambda \stds$ at any $\lambda \in \stds$,
\end{theorem}
\begin{proof}
	By Proposition \ref{prop:derivativesOfF}, $A(dx(v)) = \sigma(v)$
	and hence the first statement now follows from Proposition \ref{prop:AandAinverse}. Applying Corollary 
	\ref{cor:metrics from points est}
	we also get that there is a $\beta=\beta(n)$ such that
	\[
	  \bigabsval{dx(v) - \sigma(v)} \simleq
				C_0 h^2 (1+\beta C_0 \theta^{-2}h^2)\absval[g^e]{v},
	\]
	so we can take $\alpha'<\alpha$ small enough so that $\beta C_0 \theta^{-2}(\alpha')^2<1$ to get the second statement.
\end{proof}

We can now turn our estimate on $dx$ into an estimate
of the relevant metric tensors.
\begin{theorem}
	\label{thm:comparisongge}
	There exists $\alpha = \alpha(n,C_0)$
	such that if $g^e$ is the metric of a $(\theta,h)$-full simplex
		 	with $h<\alpha$
	then for tangent vectors
	$v,w \in T_\lambda \stds$ at any $\lambda \in \stds$
	\begin{equation}
		\label{eqn:comparisongge}
		\absval{(x^* g-g^e)(v,w)} \simleq C_0  (1+\theta^{-2}) h^2
								\absval[g^e] v \, \absval[g^e] w
	\end{equation}
	and there is an $\alpha' = \alpha'(n,\theta,C_0)$ such that if 
	 furthermore $h<\alpha'$ then
	\begin{equation}
		\label{eqn:comparisongge2}
		\absval[x^*g]v^2  \simleq \absval[g^e] v^2.
	\end{equation}	
\end{theorem}
\begin{proof}
	Due to Lemma \ref{lem:fromNormToScalarProduct},
	it suffices to show the claim for $v = w$. Consider a vector $v$ at a
	point $\lambda \in \stds$ with image $a = x(\lambda)$. We will compare
	the value of $\absval[g]{\sigma(v)}^2$.
	By  Theorem \ref{thm:EstimateFordx}, we can choose $\alpha$ small enough so that if $h\leq \alpha$ then
	\begin{align*}
		\big| \absval[x^* g] v - \absval[g]{\sigma(v)} \big|
			& = \big|\absval[g]{dx(v)} - \absval[g]{\sigma(v)} \big| \\
			& \leq \absval[g]{dx(v) - \sigma(v)}\\
			&	\simleq C_0 h^2 \absval[g]{\sigma(v)}.
	\end{align*}
	The same is true for the squared norms,
	\begin{equation} \label{eq:sg sigma est}
	  \big|\absval[x^* g] v^2 -  \absval[g]{\sigma(v)}^2 \big|
	  \simleq C_0 h^2 \absval[g]{\sigma(v)}^2
	\end{equation}
	if $\alpha$ is small enough. Hence we have successfully
	compared $x^* g$ to the Euclidean metric $\sigma^*g$. By (\ref{eq:sg sigma est}), 
	\begin{align*}
	  \big| (x^* g - g^e)(v,v) \big| &\leq
			\big| \absval[x^* g] v^2-\absval[g]{\sigma(v)}^2 \big| + \big| \absval[g]{\sigma(v)}^2- \absval[g^e]{v}^2\big| \\
			&\simleq C_0 h^2 \absval[g]{\sigma(v)}^2+ \big| \absval[g]{\sigma(v)}^2- \absval[g^e]{v}^2\big|\\
			&\simleq C_0 h^2  \absval[g^e]{v}^2 + (1+C_0 h^2)\big| \absval[g]{\sigma(v)}^2- \absval[g^e]{v}^2\big|.
	\end{align*}
	We can now use 	Corollary \ref{cor:metrics from points est} to get 
	\[
		 \big| (x^* g - g^e)(v,v) \big|	\simleq (1 + (1+C_0 h^2)\theta^{-2})C_0 h^2 \absval[g^e]{v}^2.
	\]
	The first result follows if $C_0\alpha^2<1$. The second result follows if we take $\alpha'<\alpha$ small enough so that 
	$C_0  (1+\theta^{-2}) (\alpha')^2 < 1$.
\end{proof}



\subsection{Second derivative estimates}

We are now ready to prove the estimate for $\nabla d x$.
\begin{theorem}
	\label{thm:EstimateForNabladx}
	There exists $\alpha = \alpha(n,\theta,C_0,C_1)$ such that
	if $g^e$ is the metric of a $(\theta,h)$-full simplex
	with $h < \alpha$, then
	\[
	  \absval[g]{\nabla dx(v,w)} \simleq C_{0,1} (1+\theta^{-1})h \absval[g^e] v \,
	  \absval[g^e] w,
	\] 
	or, using the operator norm,
	\[
	  \norm{\nabla dx} \simleq C_{0,1} (1+\theta^{-1})h
	\]
	 when $x$ is considered as a mapping $(\stds,g^e) \to (M,g)$.
\end{theorem}
\begin{proof}
	Throughout the proof we will assume $h<\alpha$, and show that we can make $\alpha$
	sufficiently small to prove the estimate.
	As before,
	it suffices to show the theorem for $v = w$. Let $V=dx(v)$.
	Since the $v^i$ sum up to zero,
	\begin{align*}
		\absval[g]{v^i \nabla_V X_i}=
		\absval[g]{v^i \nabla_V X_i - \sum v^i V}
			&\simleq (\theta h)^{-1}\absval[g^e]{v}\sum{ \absval[g]{\nabla_V X_i - V}} \\
			&\simleq C_0 \theta^{-1} h \absval[g^e]{v} \absval[g] V
	\end{align*}
	using Propositions \ref{prop:innerprod est} and \ref{prop:estimateForNablaX}.
	
	Now we use Proposition \ref{prop:derivativesOfF},
	\begin{align*}
	  \absval[g]{A(\nabla d x(v,v))} 
	    &\leq 2 \absval[g]{v^i \nabla_V X_i}
					+ \absval[g]{\lambda^i \nabla^2_{dx(v),dx(v)} X_i}\\
			& \simleq C_0 \theta^{-1} h \absval[g^e]{v} \absval[g] V
				+ C_{0,1} h \absval[g]{V}^2,
	\end{align*}	
	where $\nabla^2 X_i$ has been estimated by
	Proposition \ref{prop:estimateForNablaX}. Since $\absval[g]V=\absval[x^*g]{v}$, it follows from Theorem \ref{thm:comparisongge} that for $\alpha$ sufficiently small,
	\[
	  \absval[g]{A(\nabla d x(v,v))} 
	  \simleq C_{0,1} \left(1+\theta^{-1}\right) h \absval[g^e] {v}^2.
	\]
	By Proposition \ref{prop:AandAinverse},
	\begin{align*}
	  \absval[g]{\nabla d x(v,v)} &\simleq  \absval[g]{A(\nabla d x(v,v))}\\
				 &\simleq  C_{0,1} (1+\theta^{-1}) h\absval[g^e] {v}^2.
	\end{align*}
\end{proof}
The corresponding estimate for the metric is the following.
\begin{theorem}
	\label{thm:EstimateForChristoffelSymbols}
	Let the pull-back of the connection on $M$
	be defined as $dx\nabla^{x^*g}_v w = \nabla_{dx(v)} dx(w)$.
	There is a constant $\alpha = \alpha(n,\theta,C_0,C_1)$
	such that if $g^e$ is the metric of a $(\theta,h)$-full simplex
	with $h < \alpha$, then
	\[
	 \absval{\nabla^e x^*g (u,v,w)} \simleq 
	  C_{0,1} (1+\theta^{-1})h\absval[g^e] u \, \absval[g^e] v \, \absval[g^e] w
	\]
	for tangent vectors
	$u,v,w \in T_\lambda \stds$ at any $\lambda \in \stds$.
\end{theorem}
We note that  the estimate is equivalently a bound on the difference of the 
Christoffel symbols corresponding to the metrics $g^e$ and $x^*g$, which is a tensor 
(the difference of connections is a tensor).  Since the metric $g^e$ is constant,
this is an estimate for the Christoffel symbols for the metric $g$ in coordinate $x$.
\begin{proof}
%
	We see that
	\begin{align*}
	  \nabla^e x^*g (u,v,w) &= u [g(dx(v),dx(w))] - g(dx(\nabla^e _u v), dx(w)) - g(dx(v), dx(\nabla^e _u w)) \\
	  			& = g(\nabla_{dx(u)} dx(v), dx(w)) +g(dx(v), \nabla_{dx(u)} dx(w)) \\
	  			 &\qquad  - g(dx(\nabla^e _u v), dx(w)) - g(dx(v), dx(\nabla^e _u w)) \\
	  			 &=  g(\nabla dx (u,v),dx(w))  +g(dx(v), \nabla dx(u,w)),	  			  
	\end{align*}
	so Theorem \ref{thm:EstimateForNabladx} gives the estimate 
	\[
	  \absval{\nabla^e x^*g (u,v,w)} \simleq 
	   C_{0,1} (1+\theta^{-1})h \absval[g^e] u \absval[g^e] v \absval[x^*g]{w}
	   + C_{0,1} (1+\theta^{-1})h \absval[g^e] u \absval[g^e] w \absval[x^*g]{v}.
	\]
	Applying Theorem \ref{thm:comparisongge} gives the result.
\end{proof}


\section{Appendix}

We promised the reader evidence for Theorem \ref{prop:estimateForNablaX},
which is given in this appendix.

Throughout this section, we will assume the following notation.
	Let $(M,g)$ be a complete $m$-dimensional Riemannian manifold with
	curvature bounds $\norm[\infty] R \leq C_0$ and $\norm[\infty]{\nabla R} \leq C_1$,
	where $R$ denotes the curvature tensor.
	Suppose $\gamma: [0;\tau] \to M$ is the unique arclength-parametrized
	geodesic with $\gamma(0) = p$ and
	$\gamma(\tau) = q$, and $V \in T_q M$. Let
	$s \mapsto \delta(s)$ be a geodesic with $\delta(0) = \gamma(\tau)$ and
	$\dds \delta(0) = V$. Define a variation of geodesics by
	\[
		c(s,t) := \exp_p \big(\smallfrac t \tau (\exp_p)\inv \delta(s) \big)
	\]
	so $c(0,t) =\gamma(t)$.  
	Then for each small $s$ and $t\in [0,\tau]$, $T := \partial_t c$ is tangent to a geodesic curve
	and $J(s,t) := \partial_s c$ is a Jacobi field along $t \mapsto c(s,t)$. We note that 
	$\absval{T(0,t)}=1$ since $\gamma$ is parametrized by arclength.
	We denote $t$ covariant derivatives by a dot, so this means $\dot T = 0$
	and $\ddot J + R(J,T)T = 0$.

	Recall that we use $A \simleq B$ to mean there exists a constant $\beta=\beta(n)$ 
	such that $A \leq \beta B$. Since there is no simplex in this section, and hence no $n$, 
	the constant $\beta$ will be universal. In particular, the constant 
	does not depend on $m$, the dimension of $M$.


\subsection{Jacobi Fields with Two Fixed Values}

In this section we review some estimates on Jacobi fields.

The first result we need is a form of the Rauch comparison theorem. Here is a reformulation 
of part of the result \cite[Theorem 5.5.1]{Jost11}.
\begin{theorem}
\label{thm:rauch comparison}
We have the following estimate for $J(0,t)$:
\[
  g(J(0,t),\dot J(0,t)) \geq \absval{J(0,t)}^2\frac{\sqrt{C_0} \cos \sqrt{C_0}t}{\sin \sqrt{C_0}t}.
\]
In particular, $\absval{J(0,t)}$ is increasing in $t$ if $t < \frac{\pi}{2\sqrt{C_0}}$. 
\end{theorem}
\begin{proof}
	The first part is essentially the first part of \cite[Theorem 5.5.1]{Jost11}. The second follows because
	\[
	  \ddt \absval{J(0,t)} = \frac{g(J(0,t),\dot{J}(0,t))}{\absval{J(0,t)}}
	\]
    and both $\sin \sqrt{C_0}t$ and $\cos \sqrt{C_0}t$ are positive if $t < \frac{\pi}{2\sqrt{C_0}}$. 
\end{proof}
\begin{lemma}
	\label{thm:estimatesOnDtJ}
	Suppose $\tau < \frac{\pi}{2\sqrt{C_0}}$. Then
	\[
		\absval{\tau \dot J(0,\tau) - V} \simleq C_0 \tau^2 \, \absval V
	\] 
	and
	\[
			\absval{\dot J(0,\tau)} \simleq \frac 1 \tau (1+C_0 \tau^2) \absval V.
	\]
\end{lemma}
\begin{proof}
	By Theorem \ref{thm:rauch comparison} we
	have that $\absval J$ is increasing
	for all $t < \tau$. Now
	observe $J(s,0) - 0 \dot J(s,0) = 0$ and
	\[
		\left| D_t\left( J(0,t) - t \dot J(0,t) \right) \right|
			= t \, \absval{\ddot J(0,t)} \leq C_0 t \, \absval{J(0,t)}
			\leq C_0 t \absval V
	\]
	because $\absval{J(0,t)} \leq \absval{J(0,\tau)} = \absval V$.
	So $U(t) = J(0,t) - t \dot J(0,t)$
	vanishes at $t = 0$, and we have bounded its derivative.
	If $P_{a,b}$ denotes the parallel transport
	$T_{\gamma(a)}M \to T_{\gamma(b)}M$ along $\gamma=c(0,\argdot)$,
	the fundamental theorem of calculus reads for
	covariant derivatives along a geodesic:
	\[
		U(t) = P_{0,t} U(0) + \int_0^t P_{r,t} \dot U(r) \, \mathrm d r,
	\]
	as can be easily seen by taking a parallel frame along the geodesic.
	This gives
	$\absval{J(0,t) - t \dot J(0,t)} \leq \smallfrac 1 2 C_0 t^2 \absval V$.
	Since $J(0,\tau) = V$, the first claim is proven.
%
\end{proof}
\begin{lemma}
	\label{thm:estimatesOnDsJ}
	Suppose
	$C_0 \tau^2 < 1$.
	Then
	\[
		\absval{D_s \dot J(s,t)} \simleq (C_0 + \tau C_1) \absval V^2.
	\]
%
\end{lemma}
\begin{proof}
	Our approach is to derive some differential equation for
	$D_s J = \nabla_J J$, which has boundary values $D_s J (s,0) = 0$ and
	$D_s J(s,\tau) = 0$ for all $s$ because $J(s,0) = 0$ is constant
	in $s$ and $J(s,\tau) = \dds \delta(s)$ is
	the tangent of a geodesic.
	
	Because $J$ and $T$ are coordinate vector fields,
	$D_s D_t W = D_t D_s W + R(J,T)W$ for every vector field $W$. We see that 
	\[
		D_s \dot J = D_t D_s J + R(J,T)J 
	\] 
	and if we let $U=D_s J$, 
	\[
		\absval{D_s \dot J(s,t)} \leq \absval{\dot U}+ C_0 \absval T \absval{J}^2.
	\]
	We can now use the symmetries of $R$ to get
	\[
		\begin{split}
		D_s \ddot J = D_s D_t D_t J
			& = D_t D_s D_t J + R(J,T)\dot J \\
			& = D_t D_t D_s J + D_t(R(J,T)J) + R(J,T) \dot J \\
			& = D^2_{tt} (D_s J) + \dot R(J,T)J + R(\dot J,T)J + 2R(J, T)\dot J.
		\end{split}
	\]
	The (negative) left-hand side is, due to the Jacobi equation,
	\[
		-D_s \ddot J = D_s\big(R(J,T)T\big) = (D_s R)(J,T)T + R(D_s J, T)T
			+ R(J, \dot J)T + R(J,T)\dot J
	\]
	(note $D_s T = D_t J = \dot J$). From now on, we consider $J$ and $\dot J$
	as being part of the given data (which is allowed,
	as we have already sufficiently described their
	behavior). So we have a linear second-order
	ODE for $U$:
	\[
		\ddot U = A U + B,
	\]
	where both sides scale with $1 / \lambda^2$ when $t$ is replaced by $\lambda t$,
	and the operator norm of $A$ is bounded through $\norm A
	\leq C_0 \absval T^2(s)$. At $s = 0$, we have $\absval T (0) = 1$ and hence
	$\norm A \tau^2 \leq C_0 \tau^2 \leq 1$ and
%
%
	\[
		\begin{split}
		\absval B & \simleq C_1 \absval J^2 + C_0 \absval J \, \absval{\dot J} \\
			& \simleq C_1 \absval V^2 + C_0 \absval V (\smallfrac 1 \tau + C_0 \tau) \absval V \\
			& \simleq (C_0 + \tau C_1) \, \smallfrac 1 \tau \absval V^2.
		\end{split}
	\]
	(we have used Lemma \ref{thm:estimatesOnDtJ} and the assumption that $C_0 \tau^2 \leq 1$). Now consider Fermi coordinates
	to obtain an ordinary differential equation in Euclidean space: They map $c(0,\argdot)$ to
	some coordinate line, say $x_1$, along which $g$ is
	the identity matrix, and
	all Christoffel symbols vanish. Therefore,
	for any smooth vector field
	$V = \sum V^i \frac\partial{\partial x^i}$, the covariant
	derivative in direction $T = \partial_t c$
	is just $\nabla_T V = \frac{\partial V^i}{\partial x^1}
	\frac \partial {\partial x^i}$. In this frame, our ODE has the
	coordinate expression
	\[
		\sum \frac{\partial^2 U^i}{(\partial x^1)^2}
			\, \frac \partial {\partial x^i}
		= \sum \Big({\textstyle\sum A^i_j U^j} + B^i \Big) \frac \partial {\partial x^i}.
	\]
	As we only need to know the values of $U$ on $x = (t,0,\dots,0)$,
	this gives a differential equation for the components $U^i$
	of the same form as above.
	The claim on $U$ and $\dot U$ is then contained in the following lemma.
%
\end{proof}
\begin{lemma}
	\label{lem:estimateOnSecondOrderODE}
	Consider some $\Cont^2$ function $U: [0;\tau] \to \R^m$ 
	satisfying the linear second-order differential equation
	$\ddot U = AU + B$ with smooth time-dependent data $A(t) \in \R^{m \times m}$
	and $B(t) \in \R^m$ as well as boundary conditions $U(0) = U(\tau) = 0$.
	Then, provided that $\norm{A(t)} \tau^2 \leq 1$ everywhere, we have
	$\absval{\dot U(t)} \simleq \absval B \tau$.
\end{lemma}
\begin{proof}
	Denote the maxima of $\norm A$ and $\absval B$ over
	$[0,\tau]$ as $a$ and $b$ respectively.
	As $U$ is $\Cont^2$, there is an upper bound $K$ for $\absval U$ on $[0,\tau]$.
	Let $\xi_{0}$ be a point in $\left[  0,\tau\right]  $ that maximizes
	$\left\vert U\right\vert ^{2},$ i.e.,
	\[
	\left\vert U\left(  \xi_{0}\right)  \right\vert ^{2}=K^{2}.
	\]
	Since $\left\vert U\left(  0\right)  \right\vert =\left\vert U\left(
	\tau\right)  \right\vert =0,$ we must have that $\xi_{0}\in\left(
	0,\tau\right)  $ and thus at $\xi_{0}$ we have%
	\begin{align*}
	0 &  =\frac{d}{dt}\left\vert U\right\vert ^{2}=2U\cdot\dot{U}.
	\end{align*}
	
	We now see that
	\begin{align*}
	K^{2}+\xi_{0}^{2}\left\vert \dot{U}\left(  \xi_{0}\right)  \right\vert ^{2}  
	&= \left\vert U\left(  \xi_{0}\right)  -\xi_{0}\dot{U}\left(  \xi_{0}\right)
	\right\vert ^{2}\\
	& =\left\vert \int_{0}^{\xi_{0}}t \ddot U dt\right\vert ^{2}\\
	& \leq\left(  \int_{0}^{\xi_{0}}t\left(  aK+b\right)  dt\right)  ^{2}\\
	& =\left(  \frac{1}{2}\xi_{0}^{2}\left(  aK+b\right)  \right)  ^{2}.
	\end{align*}
	It then follows that
	\[
		K\leq\frac{1}{2}\xi_{0}^{2}\left(  aK+b\right)  \leq\frac{1}{2}\tau^{2}\left(
		aK+b\right)
	\]
	so
	\[
		K\leq\frac{\tau^{2}b}{2\left(  1-\frac{1}{2}\tau^{2}a\right)  }\leq\tau^{2}b
	\]
	if $\tau^{2}a\leq1.$ It also follows that
	\[
		\xi_{0}\left\vert \dot{U}\left(  \xi_{0}\right)  \right\vert \leq\frac{1}%
		{2}\xi_{0}^{2}\left(  aK+b\right),
	\]
	and so
	\[
		\left\vert \dot{U}\left(  \xi_{0}\right)  \right\vert \leq\frac{1}{2}%
		\tau\left(  ab\tau^{2}+b\right)  \leq b\tau.
	\]	
	We can now estimate $\left\vert \dot{U}\left(  \xi\right)  \right\vert $ at
	any point as
	\begin{align*}
		\left\vert \dot{U}\left(  \xi\right)  \right\vert  
		&  \leq\left\vert \int
			_{\xi_{0}}^{\xi} \ddot U dt\right\vert +\left\vert \dot{U}\left(
			\xi_{0}\right)  \right\vert \\
		&  \leq\int_{\xi_{0}}^{\xi}\left(  aK+b\right)  dt+b\tau\\
		&  \leq\tau\left(  ab\tau^{2}+b\right)  +b\tau\\
		&  \leq3b\tau.
	\end{align*}
%
\end{proof}
\begin{remark}
	If $\gamma$ is not parametrized by arclength, we introduce
	$\ell := \tau \absval T(0) = \dist(p,q)$, and the estimates become
	\[
		\absval{\tau \dot J(0,\tau) - V} \simleq C_0 \ell^2 \, \absval V,
		\qquad
		\absval{\tau D_s \dot J(0,\tau)} \simleq (C_0 + \ell C_1)
			\, \ell \, \absval V^2.
	\]
\end{remark}

\begin{remark}
The motivation for this calculation is as follows.
It is well-known that Jacobi fields grow approximately linear:
If $J(t)$ is a Jacobi field along some arclength-parametrized
geodesic $\gamma(t)$ and $P_{0,t}$ is the
corresponding parallel transport from $T_{\gamma(0)}M$ to $T_{\gamma(t)}M$ along $\gamma$, then
\begin{equation}
	\label{eqn:JacobiFieldsUsualEstimate}
	\absval{J(t) - P_{0,t}(J(0) + t \dot J(0))}
		\leq C_0 t^2(\absval{J(0)} + t \absval{\dot J(0)})
	\qquad \text{for $C_0 t^2 < \pi$}.
\end{equation}
In fact, \cite[thm. 5.5.3]{Jost11} proves that 
\[
	\absval{J(t) - P_{0,t}(J(0) + t \dot J(0))} \leq
	\absval{J(0)}(\cosh ct - 1) +
	\ddt \absval{J(0)}(\frac 1 c \sinh ct -t)
\]
for $c = \sqrt{C_0}$. By Taylor expansion and
$\ddt \absval J \leq \absval{\dot J}$,
this estimate implies (\ref{eqn:JacobiFieldsUsualEstimate}).
This is the initial-value estimate corresponding
to our boundary-value setting.
\end{remark}

\begin{remark}
The proofs of Lemmas \ref{thm:estimatesOnDtJ} and \ref{thm:estimatesOnDsJ} can be extended to
prove more general lemmas stating that if the derivatives of the curvature tensor up to order $k$ are
	bounded by constants $C_0,\dots,C_k$, then we can find bounds on $\absval{D^{k+2}_{t\dots t} J}$
	and $\absval{D^k_{s\dots s}\dot J}$. This is done by noticing that
	Theorem \ref{thm:rauch comparison} is true for $s$ different from zero as well since $J(s,\cdot)$ is a 
	Jacobi field,
	and then by differentiating the Jacobi Equation and estimating ODE. (One technical aspect of this approach
	is that we need to estimate $\absval{T}(s)$, which is close to 1. This estimate
	can be obtained using the fact that $D_s T = D_t J$	and using Lemma \ref{thm:estimatesOnDtJ}.)
\end{remark}


\subsection{Derivatives of the Squared Distance Function}

\begin{definition}
	\label{def:GradientOfSquaredDistance}
	For $p \in M$, let $\dist_p$ be the geodesic distance to $p$,
	\[
		Y_p := \grad \dist_p, \qquad X_p := \smallfrac 1 2 \grad \dist_p^2 = \dist_p Y_p.
	\]
\end{definition}
\begin{remark}
	\begin{subenum}
	\item \label{rem:YisGeodesicTangent}
	As $Y_p$ is the tangent of an arclength-parametrized geodesic,
	$\nabla_{Y_p} Y_p = 0$. Along
	this geodesic, $\dist_p$ is exactly the arclength, so $Y_p(\dist_p) = 1$
	and $X_p(\dist_p) = \dist_p$.
	
	\item The fact $V(\dist_p) = 0$ for $V \perp Y_p$ is usually referred
	to as the Gauss lemma.
	
	\item $X_p|_q$ is the tangent $\dot \gamma(1)$ of the geodesic with
	$\gamma(0) = p$ and $\gamma(1) = q$. Reversing the direction of $\gamma$
	shows $X_p|_q = - P_{0,1} X_q|_p$, where $P$ is the parallel transport along $\gamma$.
	So
	\[
		\exp_p(-X_q|_p) = q, \qquad (\exp_p)\inv(q) = -X_q|_p = P_{1,0} X_p|_q.
	\]
	\end{subenum}
\end{remark}
\begin{lemma}
	\label{prop:firstDerivOfX}
	For $V \in T_q M$, where $q$ is in a convex neighborhood of $p$,
	recall the Jacobi field $J$ introduced at the beginning of the section. Then
	\[
		\nabla_V X_p = \tau \dot J(0,\tau),
		\qquad
		\nabla^2_{V,V} X_p = \tau D_s \dot J(0,\tau).
	\]
\end{lemma}
\begin{proof}
	For the variation of geodesics $c$ inducing $J$,
	the $t$-derivative is
	\[
		\partial_t c(s,t) = \smallfrac 1 \tau P_{0,t}(\exp_p)\inv \delta(s)
			= \smallfrac 1 \tau P_{\tau,t} X_p|_{\delta(s)}
	\]
	and hence $\tau \dot J(0,\tau) = \tau D_t \partial_s c(0,\tau)
			= \tau D_s \partial_t c(0,\tau)
			= D_s X_p|_{c(0,\tau)}
			= \nabla_{J(0,\tau)} X_p$.
	Differentiating this once more gives the claim
	for the second derivative.
	If $V$ is parallel to $X_p$,
	then use $\nabla_Y Y = 0$.
\end{proof}
\begin{remark}
	Analogous to $(\exp_p)\inv(q) = -X_q|_p$,
	the derivatives of $X$ and $\exp$
	correspond to the following: $\nabla_V X_p$ is the derivative of some Jacobi field
	with prescriped start and end value, whereas $d(\exp_p V)(W)$ is the
	end value $J(1)$ of a Jacobi field $J$ along the geodesic
	$t \mapsto \exp_p tV$ with $J(0) = 0$ and $\dot J(0) = W$,
	cf. \cite[eqn. 1.2.5]{Karcher89}.
\end{remark}
We may now use the estimates in Lemmas \ref{thm:estimatesOnDtJ} and \ref{thm:estimatesOnDsJ}
to show the following.
\begin{corollary}
	\label{corol:estimateForNablaX}
	If $R$ is the Riemannian curvature tensor of $(M,g)$, assume
	$\norm R \leq C_0$ and $\norm{\nabla R} \leq C_1$ everywhere. Let $q$
	be in a convex neighborhood of $p$ with distance $\tau$ to $p$.
	Then for $V \in T_q M$, 
	we have that
	\begin{align*}
		\absval{\nabla_V X_p - V} & \simleq C_0 \tau^2 \absval{V}
		& \text{if $\tau < \frac{\pi}{2\sqrt{C_0}}$},\\
		\absval{\nabla^2_{V,V} X_p} & \simleq (C_0 + \tau C_1)\tau \, \absval{V}^2
		& \text{if also $C_0 \tau^2 < 1$.}
	\end{align*}
\end{corollary}


\bibliographystyle{amsalpha}
\bibliography{karcherMeanCoordinates}

\providecommand{\bysame}{\leavevmode\hbox to3em{\hrulefill}\thinspace}
\providecommand{\MR}{\relax\ifhmode\unskip\space\fi MR }
\providecommand{\MRhref}[2]{%
  \href{http://www.ams.org/mathscinet-getitem?mr=#1}{#2}
}
\providecommand{\href}[2]{#2}
\begin{thebibliography}{BCOS01}

\bibitem[Afs11]{Afsari11}
Bijan Afsari, \emph{Riemannian {$L^p$} center of mass: existence, uniqueness,
  and convexity}, Proc. Amer. Math. Soc. \textbf{139} (2011), no.~2, 655--673.
  \MR{2736346 (2011j:53049)}

\bibitem[Bar10]{Bartels10}
S{\"o}ren Bartels, \emph{Numerical analysis of a finite element scheme for the
  approximation of harmonic maps into surfaces}, Mathematics of Computation
  \textbf{79} (2010), no.~271, 1263--1301.

\bibitem[BCOS01]{Bertalmio01}
Marcelo Bertalmio, Li-Tien Cheng, Stanley Osher, and G.~Sapiro,
  \emph{Variational problems and partial differential equations on implicit
  surfaces}, Journal of Computational Physics \textbf{174} (2001), 759--780.

\bibitem[BDG11]{Boissonnat11}
Jean-Daniel Boissonnat, Ramsay Dyer, and Arijit Ghosh, \emph{Stability of
  {D}elaunay-type structures for manifolds}, INRIA Technical Report
  \textbf{CGL-TR-2} (2011), 229--238.

\bibitem[Blu52]{Blumenthal52}
Leonard~M. Blumenthal, \emph{Distance geometry}, Clarendon Press, Oxford, 1952.

\bibitem[BS08]{Brenner08}
Susanne~C. Brenner and L.~Ridgway Scott, \emph{The mathematical theory of
  finite element methods}, 3rd ed., Texts in Applied Mathematics, no.~15,
  Springer, 2008.

\bibitem[CM06]{Cruzeiro06}
Ana~Bela Cruzeiro and Paul Malliavin, \emph{Numerical approximation of
  diffusions in using normal charts of a {R}iemannian manifold}, Stochastic
  Processes and their Applications \textbf{116} (2006), 1088 -- 1095.

\bibitem[DE13]{Dziuk13}
Gerhard Dziuk and Charles~M. Elliott, \emph{Finite element methods for surface
  \textsc{pde}s}, Acta Numerica \textbf{22} (2013), 289--396.

\bibitem[Dey13]{Deylen14}
Stefan W.~von Deylen, \emph{Numerical approximation in {R}iemannian manifolds
  by {K}archer means}, Dissertation, Freie Universit\"at Berlin, 2013,
  \href{https://arxiv.org/abs/1505.03710}{arXiv:1505.03710}.

\bibitem[DVW15]{Dyer14}
Ramsay Dyer, Gert Vegter, and Mathijs Wintraecken, \emph{Riemannian simplices
  and triangulations}, 31st {I}nternational {S}ymposium on {C}omputational
  {G}eometry, LIPIcs. Leibniz Int. Proc. Inform., vol.~34, Schloss Dagstuhl.
  Leibniz-Zent. Inform., Wadern, 2015, pp.~255--269.

\bibitem[DVW16]{Dyer16}
\bysame, \emph{Barycentric coordinate neighbourhoods in {R}iemannian
  manifolds}, 2016, preprint,
  \href{https://arxiv.org/abs/1606.01585}{arXiv:1606.01585}.

\bibitem[DW87]{Dekster87}
Boris~V. Dekster and John~B. Wilker, \emph{Edge lengths guaranteed to form a
  simplex}, Archiv der Mathematik \textbf{49} (1987), 351--366.

\bibitem[Dzi88]{Dziuk88}
Gerhard Dziuk, \emph{Finite elements for the {L}aplace operator on arbitrary
  surfaces}, Partial Differential Equations and Calculus of Variations (Stefan
  Hildebrandt and Rolf Leis, eds.), Lecture Notes in Mathematics, no. 1357,
  Springer, 1988, pp.~142--155.

\bibitem[Fie11]{Fiedler11}
Miroslav Fiedler, \emph{Matrices and graphs in geometry}, Encyclopedia of
  Mathematics and its Applications, no. 139, Cambridge University Press,
  Cambridge, 2011.

\bibitem[GHS13]{Grohs13}
Philipp Grohs, Hanne Hardering, and Oliver Sander, \emph{Optimal a priori
  discretization error bounds for geodesic finite elements}, Preprint, RWTH
  Aachen Bericht Nr. 365, May 2013.

\bibitem[GK73]{Grove73}
Karsten Grove and Hermann Karcher, \emph{How to conjugate {$C^1$}-close group
  actions}, Mathematische Zeitschrift \textbf{123} (1973), 11--20.

\bibitem[Har15]{Hardering15}
Hanne Hardering, \emph{Intrinsic discretization error bounds for geodesic
  finite elements}, Dissertation, Freie Universit\"at Berlin, 2015.

\bibitem[HP11]{Hildebrandt11}
Klaus Hildebrandt and Konrad Polthier, \emph{On approximation of the
  {L}aplace--{B}eltrami operator and the {W}illmore energy of surfaces},
  Computer Graphics Forum (Proceedings of Eurographics) \textbf{30} (2011),
  1513--1520.

\bibitem[HPW06]{Hildebrandt06}
Klaus {Hildebrandt}, Konrad {Polthier}, and Max {Wardetzky}, \emph{On the
  convergence of metric and geometric properties of polyhedral surfaces},
  Geometriae Dedicata \textbf{123} (2006), 89--112.

\bibitem[HS12]{Holst12}
Michael Holst and Ari Stern, \emph{Geometric variational crimes: {H}ilbert
  complexes, finite element exterior calculus, and problems on hypersurfaces},
  Foundations of Computational Mathematics \textbf{12} (2012), 263--293.

\bibitem[HT04]{Huper04}
Knut H\"uper and Jochen Trumpf, \emph{Newton-like methods for numerical
  optimization on manifolds}, Signals, Systems and Computers, 2004. Conference
  Record of the Thirty-Eighth Asilomar Conference on Signals, Systems and
  Computers, vol.~1, 2004, pp.~136--139.

\bibitem[Jos11]{Jost11}
J\"urgen Jost, \emph{Riemannian geometry and geometric analysis}, 6th ed.,
  Springer, Berlin, 2011.

\bibitem[Kar77]{Karcher77}
Hermann Karcher, \emph{Riemannian center of mass and mollifier smoothing},
  Communications on Pure and Applied Mathematics \textbf{30} (1977), 509--541.

\bibitem[Kar89]{Karcher89}
\bysame, \emph{Riemannian comparison constructions}, Global Differential
  Geometry (Shiing-Shen Chern, ed.), Studies in Mathematics, no.~27,
  Mathematical Association of America, Buffalo, 1989, pp.~170--222.

\bibitem[Kar14]{Karcher14}
\bysame, \emph{Riemannian center of mass and so called {K}archer mean}, 2014,
  preprint, \href{https://arxiv.org/abs/1407.2087}{arXiv:1407.2087}.

\bibitem[Kau76]{Kaul76}
Helmut Kaul, \emph{Schranken f\"ur die {C}hristoffelsymbole}, manuscripta
  mathematica \textbf{19} (1976), 261--273.

\bibitem[Ken90]{Kendall90}
Wilfrid~S. Kendall, \emph{Probability, convexity, and harmonic maps with small
  image. {I}. {U}niqueness and fine existence}, Proc. London Math. Soc. (3)
  \textbf{61} (1990), no.~2, 371--406. \MR{1063050 (91g:58062)}

\bibitem[Ken13]{Kendall13}
\bysame, \emph{A survey of {R}iemannian centres of mass for data}, Proceedings
  59th ISI World Statistics Congress, Hong Kong, 2013, pp.~1786--1791.

\bibitem[Moa05]{Moakher05}
Maher Moakher, \emph{A differential geometric approach to the geometric mean of
  symmetric positive-definite matrices}, Siam J. Matrix Anal. Appl. \textbf{26}
  (2005), 735--747.

\bibitem[M{\"u}n07]{Muench07}
Ingo M{\"u}nch, \emph{A geometrically and materially nonlinear {C}osserat model
  (ein geometrisch und materiell nichtlineares {C}osserat-{M}odell)},
  Dissertation, Universit\"at Karlsruhe, 2007.

\bibitem[PP93]{Pinkall93}
Ulrich Pinkall and Konrad Polthier, \emph{Conputing discrete minimal surfaces
  and their conjugates}, Experimental Mathematics \textbf{2} (1993), 15--36.

\bibitem[PP00]{Polthier00}
Konrad Polthier and Elke Preu\ss, \emph{Variational approach to vector field
  decomposition}, Scientific Visualization (Proc. of Eurographics Workshop on
  Scientific Visualization), 2000.

\bibitem[Rus10]{Rustamov10}
Raif~M. Rustamov, \emph{Barycentric coordinates on surfaces}, Eurographics
  Symposium on Geometry Processing \textbf{29} (2010), 1507--1516.

\bibitem[San12]{Sander12}
Oliver Sander, \emph{Geodesic finite elements on simplicial grids},
  International Journal of Numerical Methods in Engineering \textbf{92} (2012),
  999--1025.

\bibitem[San13]{Sander13}
\bysame, \emph{Geodesic finite elements of higher order}, Preprint, RWTH Aachen
  Bericht Nr. 356, January 2013.

\bibitem[San16]{Sander16}
\bysame, \emph{Test function spaces for geometric finite elements}, 2016,
  preprint, \href{https://arxiv.org/abs/1607.07479}{arXiv:1607.07479}.

\end{thebibliography}

\end{document}